\documentclass[11pt]{article}

\usepackage{lineno}
\usepackage{amssymb}
\usepackage{amsmath}
\usepackage{amsthm}
\usepackage{epsfig}
\usepackage{graphicx}
\usepackage{graphics}
\usepackage{float}
\usepackage{subfigure}
\usepackage{multirow}
\usepackage{color}
\usepackage{fullpage}
\usepackage[normalem]{ulem} 
\usepackage{makeidx}
\usepackage{wrapfig}

\numberwithin{equation}{section}

\begin{document}

\theoremstyle{remark}
\newtheorem{remark}{Remark}[section]
\newtheorem{theorem}{Theorem}[section]
\newtheorem{corollary}[theorem]{Corollary}
\newtheorem{lemma}[theorem]{Lemma}
\newtheorem{proposition}[theorem]{Proposition}

\title{\textbf{Maximum Principle Preserving Finite Difference Scheme for 1-D Nonlocal-to-Local Diffusion Problems}}
 
\author{
{Amanda Gute\footnote{Department of Mathematics, University of North Carolina at Charlotte, Email: agute@uncc.edu}} \and {Xingjie Helen Li\footnote{Department of Mathematics, University of North Carolina at Charlotte, Email: xli47@uncc.edu}}
}


\maketitle
\setcounter{page}{1}

\begin{abstract}
\noindent
In a recent paper \cite{du2018quasinonlocal}, a quasi-nonlocal coupling method was introduced to seamlessly bridge a nonlocal diffusion model with the classical local diffusion counterpart in a one-dimensional space. The proposed coupling framework removes interfacial inconsistency, preserves the balance of fluxes, and satisfies the maximum principle of diffusion problem. However, the numerical scheme proposed in that paper does not maintain all of these properties on a discrete level. In this paper we resolve this issue by proposing a new finite difference scheme that ensures the balance of fluxes and the discrete maximum principle. We rigorously prove these results and provide the stability and convergence analyses accordingly. In addition, we provide the Courant–Friedrichs–Lewy (CFL) condition for the new scheme and test a series of benchmark examples which confirm the theoretical findings.
\end{abstract}

\noindent
\textbf{Keywords:}
 Nonlocal Diffusion Problem, Quasi-Nonlocal Coupling, Discrete Maximum principle, Convergence Analysis
\section{Introduction}
 Since the last decade, nonlocal integro-differential type models have been employed to describe physical systems, due to their natural ability to model physical phenomena at small scales and their reduced regularity requirements which lead to greater flexibility
\cite{Bates1999,Fife2003,Bobaru2010b,Silling2000,Chasseigne2006a,Du2012a,Du2013a,NewsDu,zhou2010a,Planas2002,Silling2005,Bobaru2010a,Bobaru2011a,Kriventsov,Lipton2014a,Lipton2016a,Silling2008,Lehoucq2010}.
These nonlocal models are defined through a length scale parameter $\delta$, referred to as a horizon, which measures the extent of nonlocal interaction. 
  An important feature of nonlocal models is that they restore the corresponding classical partial differential equation (PDE) models as the horizon $\delta \rightarrow 0$ \cite{Du2012a,Du2013a}.

\vspace{.1in}
\noindent
Nonlocal models that are compatible with the local PDEs are often much computationally expensive and require additional attention to the boundary treatments since a layer of volumetric boundary conditions is needed within the physical system. Meanwhile, nonlocal models need less regularity requirements which helps the descriptions near defects and singularities. Consequently, tremendous efforts have been devoted to combining nonlocal and local methods to keep accuracy around the irregularity while retain efficiency away from the singularity (see the review paper \cite{d2019review} for the state-of-art).

\noindent
In \cite{du2018quasinonlocal},   
a quasi-nonlocal (QNL) coupling method was proposed to combine the nonlocal and local diffusion operators in a seamless way using the variational approach. The coupled operator is proved to preserve many mathematical and physical properties on the continuous level, including the symmetry of operator, the balance of linear momentum, and the maximum principle. However, it is not clear how to retain these desired properties with proper numerical discretization. In this paper, we propose a new finite difference method which inherits all properties from the continuous case. 


\vspace{.1in}
\noindent
We recall that the linear local diffusion model in one-dimensional space is
\begin{equation}\label{LDiffMod}
    u_t(x,t)=u_{xx}(x,t)+f(x,t).
\end{equation}
The corresponding counterpart in the nonlocal setting is the linear nonlocal diffusion model which reads
\begin{equation}\label{NLDiffMod}
    \displaystyle{u_t(x,t)=\int_{-\delta}^{\delta}\gamma_{\delta}(s)\bigg(u(x+s,t)-u(x,t)\bigg)ds},
\end{equation}
where $\gamma_{\delta}(s)$ denotes the isotropic nonlocal diffusion kernel satisfying
the following convenient assumption 
with $\gamma_\delta(\cdot)$ being a rescaled kernel,
\begin{equation}\label{NLKer}
\left \{
\begin{aligned}
& \gamma_\delta(|s|)=\frac{1}{\delta^{3}} \gamma\left(\frac{|s|}{\delta}\right), \quad \gamma \text{ is nonnegative and nonincreasing on (0,1)},\\
& \text{with } \text{supp}(\gamma)\subset [0,1]  \text{ and }  \int_{-\delta}^{\delta} |s|^2 \gamma(|s|) ds =1\,.
\end{aligned}
\right.
\end{equation}
We will display more details about the coupling and numerical schemes in the following sections.

\vspace{.1in}
\noindent
More precisely, We will organize the paper as follows, In section~\ref{sec:QNL_FDM}, we recall the energy-based QNL coupling from \cite{du2018quasinonlocal} to build the coupling operator $\mathcal{L}^{qnl}_{\delta}$ bringing the nonlocal and local diffusion problems and introduce space-time discretizations as well as the new finite difference method (FDM). In section~\ref{sec:consistency}, we estimate the consistency errors of the proposed scheme using Taylor expansions. In section~\ref{sec:Stability}, we prove the discrete maximum principle and hence the stability of proposed scheme. In section~\ref{sec:convergence}, we combine the consistency and stability results to conclude the convergence estimates. In section~\ref{sec:CFL}, we mathematically study the Courant–Friedrichs–Lewy (CFL) condition for the space-time discretization. In section~\ref{sec:numerics}, we test several benchmark examples to confirm our theoretic findings. 






\section{QNL Coupling and Finite Difference Scheme}\label{sec:QNL_FDM}
Now, we consider the domain to be $\Omega_\delta=[-1-\delta,\,1]$, with the coupling interface of nonlocal and local models at $x^*=0$; $(-1,\,0)$ denotes the nonlocal region with nonlocal boundary layer at $[-1-\delta, -1]$ and $(0,\,1)$ denotes the local region with local boundary point at $\{1\}$, as illustrated in Figure~\ref{Fig:1D_domain}.

\begin{figure}[H]
\begin{center}
\includegraphics[height = 1in, width = 4.5in]{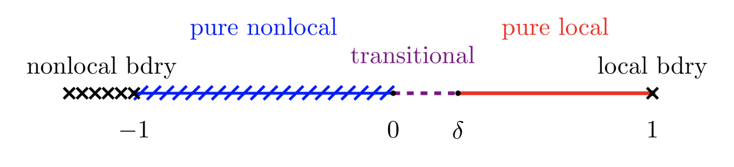}
\caption{Graphical illustration of 1-D Domain.\label{Fig:1D_domain}}
\end{center}
\end{figure}

\noindent 
In \cite{du2018quasinonlocal}, the QNL operator $\mathcal{L}^{qnl}_\delta u(x,t)$ is introduced to smoothly bridge the local and nonlocal regions over the transitional region $[0,\,\delta]$. The corresponding coupled diffusion problem is proved to be a well-posed initial value problem and is given by
\begin{equation}\label{QNL_Cont}
    \begin{cases}
      u_t(x,t)=\mathcal{L}^{qnl}_\delta u(x,t)+f(x,t), & \text{for}\hspace{.1in} T>t>0 \hspace{.1in} \text{and} \hspace{.1in} x\in(-1,1), \\
      u(x,0) = u_0(x), & \text{for}\hspace{.1in} x\in(-1,1), \\
      u(x,t) = 0, & \text{for}\hspace{.1in} x\in[-1-\delta,-1], \hspace{.1in} \text{or} \hspace{.1 in} x=1. \\
   \end{cases}
\end{equation}
$\mathcal{L}^{qnl}_{\delta}$ employed in equation \eqref{QNL_Cont} is the quasi-nonlocal coupling operator which describes the diffusion within the nonlocal, transitional, 
and local regions, respectively. The expression of $\mathcal{L}^{qnl}_{\delta}$ is given below
\begin{equation}\label{QNL_op}
\mathcal{L}^{qnl}_{\delta}u(x,t)= \begin{cases}
      \displaystyle{\int_{-\delta}^{\delta}}\bigg(u(x+s,t)-u(x,t)\bigg)\gamma_\delta(s)ds, \hspace{.2in} \text{if} \hspace{.1in} x\in(-1,0),\\
      \\
      \displaystyle{\int_{x}^{\delta}\gamma_\delta(s)\bigg(u(x-s,t)-u(x,t)\bigg)ds+\bigg{(}\int_{x}^{\delta}s\gamma_{\delta}(s)ds\bigg{)}u_x(x,t)}\\
      \displaystyle{\quad+\bigg{(} \int_{0}^xs^2\gamma_{\delta}(s)+x\int_{x}^\delta s\gamma_{\delta}(s)ds\bigg{)}u_{xx}
      (x)}, \hspace{.2in} \text{if} \hspace{.1in} x\in[0,\delta],\\
      \\
      u_{xx}(x,t), \hspace{.2in}\text{if}\hspace{.1in} x\in(\delta,1).
   \end{cases}
\end{equation}
Next, we discuss the numerical settings for the spatial and temporal discretization. We use $u_i^n$ to denote the numerical approximation of the exact solution $u(x_i, t^n)$ with spatial and temporal step sizes being with $\Delta x:=\frac{1}{N}$ and $\Delta t:=\frac{T}{N_T}$, respectively. Hence, the spatial grid is $x_i$ and temporal grid is $t_n=n\Delta t$. 
For simplicity, we drop $x$ and $t$ but only use $i$ and $n$ accordingly. The relation between $\Delta x$ and $\Delta t$ will be determined later by the CFL condition. 
Meanwhile, we assume that the horizon $\delta$ is a multiple of $\Delta x$ with $\delta=r\Delta x$ and $r\in \mathbb{N}$. 

\vspace{.1in}
\noindent
Recall that the entire computational domain is $\Omega_\delta:=[-1-\delta, \,1]$, so the interior domain is $\Omega=[-1,1]$ with interface at $x^*=0$; the volumetric boundary layer for the nonlocal region is $\Omega_{n}=[-1-\delta,-1)$; and the local boundary point is $\Omega_c=\{1\}$. Next we denote the set of spatial grids by $I$ and ${I}=I_{\Omega}\cup I_{\Omega_{n}}\cup I_{\Omega_{c}}$, where $I_{\Omega}=\{1,2,...,2N-1\}$ denotes the interior grids, $I_{\Omega_{n}}=\{-(r-1),...,0\}$ denotes the nonlocal volumetric boundary grids, and $I_{\Omega_{c}}=\{2N\}$ denotes the local boundary point. 
Following the scope of asymptotically compatible scheme \cite{Tian2013a,Tian2014a}, we define the spatial discretization of the QNL coupling operator $\mathcal{L}_{\delta,\Delta x}^{qnl}$ as follows
\begin{equation}\label{QNL_FDM}
\mathcal{L}_{\delta,\Delta x}^{qnl}u_i^n:= \begin{cases}
      \displaystyle{\sum_{j=1}^r\frac{u_{i+j}^n-2u_i^n+u_{i-j}^n}{(j\Delta x)^2}\int_{(j-1)\Delta x}^{j\Delta x}s^2\gamma_{\delta}(s)ds}, \hspace{.25in} \text{if} \hspace{.1in} x_i<0,\\
      \\
      \displaystyle{\sum_{j=\frac{x_{i}}{\Delta x}+1}^{r}\frac{u_{i+j-1}^n-2u_{i}^n+u_{i-j+1}^n}{2({j-1})\Delta x}\int_{(j-1)\Delta x}^{j\Delta x}s\gamma_{\delta}(s)ds}\\
      \displaystyle{\quad -\sum_{j=\frac{x_{i}}{\Delta x}+1}^{r}\frac{u_{i+j-1}^n-u_{i-j+1}^n}{2({j-1})\Delta x}\int_{(j-1)\Delta x}^{j\Delta x}s\gamma_{\delta}(s)ds}\\
      \displaystyle{\qquad +\bigg{(}\int_{x_{i}}^{\delta}s\gamma_{\delta}(s)ds\bigg{)}\frac{u_{i+1}^n-u_{i}^n}{\Delta x}}\\
      \displaystyle{\quad\qquad +\bigg{(}\int_{0}^{x_{i}}s^2\gamma_{\delta}(s)ds+x_{i}\int_{x_{i}}^{\delta}s\gamma_{\delta}(s)ds\bigg{)}\frac{u_{i+1}^n-2u_{i}^n+u_{i-1}^n}{(\Delta x)^2}},\hspace{.2in}\text{if} \hspace{.1in} x_i\in[0,\delta],\\
      \\
      \displaystyle{\frac{u_{i+1}^n-2u_{i}^n+u_{i-1}^n}{(\Delta x)^2}}, \hspace{1.55in}\text{if}\hspace{.1in} x_i\in(\delta,1).\\
   \end{cases}
\end{equation}
For the temporal discretization, we employ the simplest explicit Euler scheme due to the limitation of first order accuracy in the spatial discrezation, which will be proved later. Hence the full FDM discretization of \eqref{QNL_Cont} is \begin{equation}\label{QNL_scheme}
\frac{u_i^{n+1}-u_i^n}{\Delta t}=
\mathcal{L}_{\delta,\Delta x}^{qnl}u_i^n+f^n_i,\quad i\in I_{\Omega},
\end{equation}
where $f^n_i=f(x_i,t^n)$.


\vspace{.1in}
\noindent
Figure~\ref{Fig:stencil} displays a sampling set of spatial stencils using $N=5$ on domain $[-1-\delta,1]$.  The step size is $\Delta x = \frac{1}{5}$ and the horizon $\delta =r\Delta x$ with $r=3$. 
\begin{figure}[H]
\begin{center}
 \includegraphics[height = 1.1in, width = 5.25in]{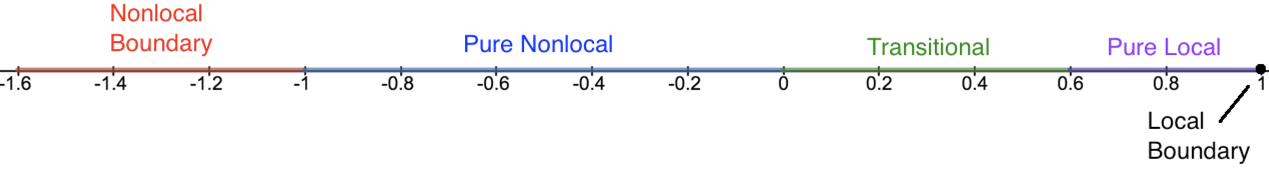}
\caption{Illustration of the finite difference stencil.\label{Fig:stencil}}
\end{center}
\end{figure}

\section{Consistency}\label{sec:consistency}
In this section, we estimate the consistency error of the scheme \eqref{QNL_scheme} with $\mathcal{L}^{qnl}_{\delta, \Delta x}$ defined in \eqref{QNL_FDM}.



\begin{theorem}\label{theorem:consistency}
Let the horizon $\delta=r\Delta x$ with $r\in\mathbb{N}$ and suppose $u(x,t)$ is the strong solution to \eqref{QNL_Cont}, and $u_i^n$ is the discrete solution to the scheme \eqref{QNL_scheme} with
$i\in I_{\Omega}$ and $t^n=n\Delta t$. Also assume that the exact solution $u$ is sufficiently smooth, specifically
$u(x,t)\in C^{2}([-1-\delta, 1]\times [0, T])$. 
Suppose at any given time level $t^n=n \Delta t$ we have 
$u(x_i,t^n)=u_i^n,\; \forall i\in I_{\Omega}=\{1,\dots, 2N-1\}$,
then for the next time level $n+1$
the consistency error of the scheme satisfies
\begin{equation}\label{consistency_error}
    |u_i^{n+1}-u(x_i,t^{n+1})|\le C_{\delta}\Delta t\left( (\Delta x)+(\Delta t)\right),\; \forall i=1,\dots, 2N-1,
\end{equation}
where $C_{\delta}$ is a constant independent of $\Delta x$ and $\Delta t$.
\end{theorem}
\begin{proof}
We evolve $u(x_i, t^n)$ and $u^n_i$ by one time step $\Delta t$ according to three differential regions.

\vspace{.25cm}
\noindent \textbf{Local:} If $x_i>\delta$ or simply $i\in \{{N+r+1},...,{2N-1}\}$, then the continuous and discrete equations follow the expressions in the local region. So at $(x_i, t^n)$, we have 
the continuous equation: 
\begin{equation}\label{ContLEqn}
    u_t(x_i,t_n)=u_{xx}(x_i,t_n) +f(x_i,t_n),
\end{equation}
and the discrete equation: 
\begin{equation}\label{DiscLEqn}
    \frac{u_i^{n+1}-u_i^n}{\Delta t}=\frac{u_{i+1}^n-2u_i^n+u_{i-1}^n}{(\Delta x)^2}+f_i^n
\end{equation}
with $f_i^n=f(x_i, t^n)$.

\vspace{.25cm}
\noindent
Notice from consistency assumption that $u_i^n=u(x_i, t^n)$, so can rewrite the discrete equation as
\begin{equation}\label{DiscLEqn_2}
\frac{u_i^{n+1}-u(x_i, t^n)}{\Delta t}=\frac{u(x_{i+1}, t^n)-2u(x_i, t^n)+u(x_{i-1}, t^n)}{(\Delta x)^2}+f(x_i,t^n).
\end{equation}

\noindent
We apply the Taylor expansion at the spatial grid $(x_i)$ up to fourth order derivative and get an estimate of $u_i^{n+1}$, which is
\begin{align}\label{disc_local}
u_i^{n+1}=&
u(x_i, t^n)+\Delta t\left(\frac{u(x_{i+1}, t^n)-2u(x_i, t^n)+u(x_{i-1}, t^n)}{(\Delta x)^2}+f(x_i,t^n)\right)\nonumber\\
=&u(x_i, t^n)+\Delta t\left(\frac{(\Delta x)^2u_{xx}(x_i, t^n)+O(\Delta x^4)}{(\Delta x)^2}+f(x_i,t^n)\right)\nonumber\\
=&u(x_i, t^n)+\Delta t\bigg(u_{xx}(x_i, t^n)+f(x_i,t^n)\bigg)+O\big(\Delta t(\Delta x)^2\big).
\end{align}
Now, let us estimate the continuous solution $u(x_i, t^{n+1})$. This time, we apply Taylor expansion at the time grid $(t^n)$ and get
\begin{align}\label{cont_local}
u(x_i,t^{n+1})=&
u(x_i,t^n)+\Delta t u_{t}(x_i,t^n)+O(\Delta t^2)\nonumber\\
=&u(x_i,t^n)+\Delta t \bigg[\big(u_{xx}(x_i, t^n)+f(x_i,t^n)\big)\bigg]+O(\Delta t^2),
\end{align}
where we substitute $u_{t}(x_i,t^n)$ by the continuous equation on the local region.

\noindent
By subtracting \eqref{disc_local} from \eqref{cont_local} we can get
\begin{equation}\label{LocD_Cerror}
u_i^{n+1}-u(x_i,t^{n+1})=O\big(\Delta t (\Delta x)^2\big)+O\big((\Delta t)^2\big).
\end{equation}

\noindent
\textbf{Nonlocal:} Next we consider the fully nonlocal region where $x_i<0$ or simply $i\in \{1,\dots, N\}$. We first have the continuous equation: 
\begin{align}\label{NLconteqn}
u_t(x_i&,t^n)=\int_{-\delta}^{\delta}\gamma_{\delta}(s)\bigg(u(x_i+s,t^n)-u(x_i,t^n)\bigg)ds+f(x_i,t^n)\nonumber\\
=&\int_{-\delta}^0\gamma_{\delta}(s)\bigg{(}u(x_i+s,t^n)-u(x_i,t^n)\bigg{)}ds+\int_0^{\delta}\gamma_{\delta}(s)\bigg{(}u(x_i+s,t^n)-u(x_i,t^n)\bigg{)}ds+f(x_i,t^n)\nonumber\\
=&\int_0^{\delta}\gamma_{\delta}(-s)\bigg{(}u(x_i-s,t^n)-u(x_i,t^n)\bigg{)}ds+\int_0^{\delta}\gamma_{\delta}(s)\bigg{(}u(x_i+s,t^n)-u(x_i,t^n)\bigg{)}ds +f(x_i,t^n).
\end{align}

\vspace{.25cm}
\noindent
Because of the isotropic property of the nonlocal kernel $\gamma_{\delta}(s)$ summarized in \eqref{NLKer},
we have 
\begin{align}\label{NLconteqn_2}
u_{t}(x_i,t^n)=\int_0^{\delta}\gamma_{\delta}(s)\bigg{(}u(x_i+s,t^n)-2u(x_i,t^n)+u(x_i-s,t^n)\bigg{)}ds+f(x_i,t^n).
\end{align}

\noindent
Clearly, we can divide the integral into the sum of subintegrals on the union of subintervals, so we have, 
\begin{align}\label{NLconteqn_3}
u_t(x_i,t^n)=\sum_{j=1}^r\int_{(j-1)\Delta x}^{j\Delta x}\gamma_{\delta}(s)\bigg{(}u(x_i+s,t^n)-2u(x_i,t^n)+u(x_i-s,t^n)\bigg{)}ds+f(x_i,t^n).
\end{align}

\vspace{.25cm}
\noindent
Meanwhile, we have the discrete equation to advance $u_i^n$ to $u_{i}^{n+1}$:
\begin{equation}\label{discNLeqn}
\frac{u_i^{n+1}-u_i^n}{\Delta t}=\sum_{j=1}^r\frac{u_{i+j}^n-2u_i^n+u_{i-j}^n}{(j\Delta x)^2}\int_{(j-1)\Delta x}^{j\Delta x}s^2\gamma_{\delta}(s)ds+f_i^n.
\end{equation}

\noindent
Which gives,
\begin{equation}\label{discNLeqn2}
u_i^{n+1}=u_i^n+\Delta t\bigg(\sum_{j=1}^r\frac{u_{i+j}^n-2u_i^n+u_{i-j}^n}{(j\Delta x)^2}\int_{(j-1)\Delta x}^{j\Delta x}s^2\gamma_{\delta}(s)ds+f_i^n\bigg).    
\end{equation}

\noindent
Now we want to estimate the continuous solution $u(x_i,t^{n+1})$. We know that
\begin{equation}\label{contNL_00}
u(x_i,t^{n+1})=u(x_i,t^n)+\Delta t u_t(x_i,t^n)+O(\Delta t^2),
\end{equation}

\noindent
Hence, plugging the continuous description of nonlocal diffusion \eqref{NLconteqn_3}, we get
\begin{align}\label{contNL}
u(x_i,t^{n+1})&=
u(x_i,t^n)+\Delta t u_{t}(x_i,t^n)+O(\Delta t^2)\nonumber\\
&=u(x_i,t^n)+\Delta t\bigg[\sum_{j=1}^r\int_{(j-1)\Delta x}^{j\Delta x}\gamma_{\delta}(s)s^2\bigg{(}\frac{u(x_i+s,t^n)-2u(x_i,t^n)+u(x_i-s,t^n)}{s^2}\bigg{)}ds\nonumber\\
&\qquad +f(x_i,t^n)\bigg]+O(\Delta t^2).
\end{align}

\vspace{.25in}
\noindent
For each integral term from $[(j-1)\Delta x, \, j\Delta x]$ within the summation , we then focus on the fractional term and apply Taylor expand to $u(x_i+s,t^n)$ and $u(x_i-s,t^n)$ for $s$ at $(j\Delta x)$ up to fourth order derivative.
This gives an estimate of
\begin{align}\label{cont_nonlocal}
u(x_i,t^{n+1})&= u(x_i,t^{n})\nonumber\\
&+{\Delta t}\Bigg[\sum_{j=1}^r\int_{(j-1)\Delta x}^{j\Delta x}\gamma_{\delta}(s) s^2 \frac{1}{(j\Delta x)^2}\bigg( \big(u(x_{i+j},t^n)-2u(x_{i},t^n)+u(x_{i-j},t^n)\big)+O(s^4)\bigg) ds\nonumber\\
&\qquad\qquad\qquad  +f(x_i,t^n)\Bigg]+O(\Delta t^2)\nonumber\\
&=u_i^n
+{\Delta t}\Bigg[\sum_{j=1}^r\int_{(j-1)\Delta x}^{j\Delta x}\gamma_{\delta}(s) s^2 \frac{1}{(j\Delta x)^2}\bigg( \big(u_{i+j}^n-2u_{i}^n+u_{i-j}^n\big)\bigg) ds +O(\Delta x^2)\nonumber\\
&\qquad\qquad\qquad  +f(x_i,t^n)\Bigg]+O(\Delta t^2).
\end{align}


\noindent
Then by subtracting \eqref{discNLeqn2} from \eqref{cont_nonlocal}, we can get

\begin{equation}\label{NLerror}
u_i^{n+1}-u(x_i,t^{n+1})=O(\Delta t)\cdot O(\Delta x)^2+O(\Delta t^2).
\end{equation}

\vspace{.25cm}
\noindent
\textbf{Transitional:} Finally we consider when  $x_i\in[0,\, \delta]$ or equivalently $i\in\{{N+1},\dots,{N+r}\}$, and again we will look at the continuous equation for the time derivative $u_t(x_i,t^n)$ first.
\begin{align}\label{ContTEqn}
u_t(x_i,t^n)&= \bigg{[}\int_{x_i}^{\delta}\gamma_{\delta}(s)\bigg(u(x_i-s,t^n)-u(x_i,t^n)\bigg)ds+\Bigg{(}\int_{x_i}^{\delta}s\gamma_{\delta}(s)ds\Bigg{)}u_x(x_i,t^n) \nonumber\\
&+\Bigg{(}\int_{0}^{x_i}s^2\gamma_{\delta}(s)ds+x_i\int_{x_i}^{\delta}s\gamma_{\delta}(s)ds\Bigg{)}u_{xx}(x_{i},t^n)\bigg{]}+f(x_i,t^n),
\end{align}

\noindent
and splitting and symmetrizing the first integral gives
\begin{align}\label{ContTEqn2}
u_t(x_i,t^{n})&=\int_{x_i}^{\delta}\frac{\gamma_{\delta}(s)}{2}\bigg{(}u(x_i-s,t^n)-2u(x_i,t^n)+u(x_i+s,t^n)\bigg{)}ds\nonumber\\
&+\int_{x_i}^{\delta}\frac{\gamma_{\delta}(s)}{2}\bigg{(}u(x_i-s,t^n)-u(x_i+s,t^n)\bigg{)}ds+\Bigg{(}\int_{x_i}^{\delta}s\gamma_{\delta}(s)ds\Bigg{)}u_x(x_i,t^n)\nonumber\\
&+\Bigg{(}\int_{0}^{x_i}s^2\gamma_{\delta}(s)ds+x_i\int_{x_i}^{\delta}s\gamma_{\delta}(s)ds\Bigg{)}u_{xx}(x_{i},t^n)+f(x_i,t^n), 
\end{align}

\noindent
and dividing these two integrals into the sum of subintegrals on the union of subintervals, and modify each integrand in the scope of asymptotically compatible scheme \cite{Tian2014a}, we get
\begin{align}\label{ContTEqn3}
u_t(x_i,t^{n})&=\sum_{j=\frac{x_{i}}{\Delta x}+1}^{r}\int_{(j-1)\Delta x}^{j\Delta x}\frac{\gamma_{\delta}(s)s}{2}\bigg{(}
\frac{
u(x_i-s,t^n)-2u(x_i,t^n)+u(x_i+s,t^n)}{s}\bigg{)}ds\nonumber\\
&+\sum_{j=\frac{x_{i}}{\Delta x}+1}^{r}\int_{(j-1)\Delta x}^{j\Delta x}\frac{\gamma_{\delta}(s) s}{2}\bigg{(}\frac{u(x_i-s,t^n)-u(x_i+s,t^n)}{s}\bigg{)}ds+\Bigg{(}\int_{x_i}^{\delta}s\gamma_{\delta}(s)ds\Bigg{)}u_x(x_i,t^n)\nonumber\\
&+\Bigg{(}\int_{0}^{x_i}s^2\gamma_{\delta}(s)ds+x_i\int_{x_i}^{\delta}s\gamma_{\delta}(s)ds\Bigg{)}u_{xx}(x_{i},t^n)+f(x_i,t^n).  
\end{align}



\vspace{.25cm}
\noindent
Now working with the discrete equation for $u_i^{n+1}$
\begin{align}\label{DiscTEqn}
\frac{u_i^{n+1}-u_i^n}{\Delta t}&=\sum_{j=\frac{x_{i}}{\Delta x}+1}^{r}\frac{u_{i+j-1}^n-2u_{i}^n+u_{i-j+1}^n}{2({j-1})\Delta x}\int_{(j-1)\Delta x}^{j\Delta x}s\gamma_{\delta}(s)ds \nonumber\\
&-\sum_{j=\frac{x_{i}}{\Delta x}+1}^{r}\frac{u_{i+j-1}^n-u_{i-j+1}^n}{2({j-1})\Delta x}\int_{(j-1)\Delta x}^{j\Delta x}s\gamma_{\delta}(s)ds+\bigg{(}\int_{x_{i}}^{\delta}s\gamma_{\delta}(s)ds\bigg{)}\frac{u_{i+1}^n-u_{i}^n}{\Delta x}\nonumber\\
&+\bigg{(}\int_{0}^{x_{i}}s^2\gamma_{\delta}(s)ds+x_{i}\int_{x_{i}}^{\delta}s\gamma_{\delta}(s)ds\bigg{)}\frac{u_{i+1}^n-2u_{i}^n+u_{i-1}^n}{(\Delta x)^2}+f_i^n.
\end{align}

\noindent
Which gives,
\begin{align}\label{DiscTEqn1}
u_i^{n+1}&=u_i^n+\Delta t\bigg[\sum_{j=\frac{x_{i}}{\Delta x}+1}^{r}\frac{u_{i+j-1}^n-2u_{i}^n+u_{i-j+1}^n}{2({j-1})\Delta x}\int_{(j-1)\Delta x}^{j\Delta x}s\gamma_{\delta}(s)ds \nonumber\\
&-\sum_{j=\frac{x_{i}}{\Delta x}+1}^{r}\frac{u_{i+j-1}^n-u_{i-j+1}^n}{2({j-1})\Delta x}\int_{(j-1)\Delta x}^{j\Delta x}s\gamma_{\delta}(s)ds+\bigg{(}\int_{x_{i}}^{\delta}s\gamma_{\delta}(s)ds\bigg{)}\frac{u_{i+1}^n-u_{i}^n}{\Delta x}\nonumber\\
&+\bigg{(}\int_{0}^{x_{i}}s^2\gamma_{\delta}(s)ds+x_{i}\int_{x_{i}}^{\delta}s\gamma_{\delta}(s)ds\bigg{)}\frac{u_{i+1}^n-2u_{i}^n+u_{i-1}^n}{(\Delta x)^2}+f_i^n\bigg].
\end{align}

\noindent
Again we want to estimate difference between $u(x_i,t^{n+1})$ and $u_i^{n+1}$. 


\vspace{.1in}
\noindent
For each integral term $[(j-1)\Delta x, j\Delta x]$ within the summation of \eqref{ContTEqn3}, we then Taylor expand $u(x_i+s,t^n)$ and $u(x_i-s,t^n)$ for $s$ at $(j-1)\Delta x$, which is similar to the processing we did for the nonlocal region.

\begin{align}\label{ContT1}
u(x_i,&t^{n+1})= u(x_i,t^{n})\nonumber\\
&+{\Delta t}\Bigg[\sum_{j=\frac{x_{i}}{\Delta x}+1}^r\int_{(j-1)\Delta x}^{j\Delta x} \frac{\gamma_{\delta}(s)s}{2(j-1)\Delta x}\bigg( u(x_{i+j-1},t^n)-2u(x_{i},t^n)+u(x_{i-j+1},t^n)+O(s^2)\bigg) ds\nonumber\\
&+\sum_{j=\frac{x_{i}}{\Delta x}+1}^{r}\int_{(j-1)\Delta x}^{j\Delta x}\frac{\gamma_{\delta}(s)s}{2(j-1)\Delta x}\bigg{(}u(x_{i+j-1},t^n)-u(x_{i-j+1},t^n)+O(s)\bigg{)}ds\nonumber\\
&+\Bigg{(}\int_{x_i}^{\delta}s\gamma_{\delta}(s)ds\Bigg{)}\left(\frac{u(x_{i+1},t^n)-(x_i,t^n)}{\Delta x}+O(\Delta x)\right)\nonumber\\
&+\Bigg{(}\int_{0}^{x_i}s^2\gamma_{\delta}(s)ds+x_i\int_{x_i}^{\delta}s\gamma_{\delta}(s)ds\Bigg{)}\left(\frac{u_(x_{i+1},t^n)-2_(x_{i},t^n)+_(x_{i-1},t^n)}{\Delta x^2}+O(\Delta x^2)\right)\bigg.
\\&\qquad \qquad +f(x_i,t^n)\bigg]+O(\Delta t^2).
\end{align}

\noindent
By subtracting \eqref{DiscTEqn1} from \eqref{ContT1}  we can get
\begin{align}\label{TError}
u_i^{n+1}-u(x_i,t^{n+1})=O(\Delta t) O(\Delta x)+O(\Delta t^2).
\end{align}

\noindent Therefore, $\|u(x_i,t^{n+1})-u_i^{n+1}\|_{L^\infty}=
O(\Delta t)O(\Delta x)+O(\Delta t^2)$ with highest restrictions from the transitional region. Since the order of accuracy is greater than zero, the finite difference scheme is consistent.
\end{proof}

\section{Stability} \label{sec:Stability}
Global stability of the scheme is attained by the discrete maximum principle.  To prove the discrete maximum principle for the quasi-nonlocal coupling equation with an underlying finite difference discretization the spatial operator $(-\mathcal{L}_{\delta,\Delta x}^{qnl})$ must be positive-definite, and the time discretization, that is a single explicit Euler, must be a convex scheme. Recall the interior domain $\Omega=[-1,1]$ with interface at $x^*=0$.  The volumetric boundary layer for the nonlocal region is $\Omega_{n}=(-1-\delta,-1]$, and the local boundary point is $\Omega_c=\{1\}$. The corresponding sets of spatial grids are $I_{\Omega}=\{1,2,...,2N-1\}$ for $\Omega$, $I_{\Omega_{n}}=\{-(r-1),...,0\}$ for $\Omega_n$, and $I_{\Omega_{c}}=\{2N\}$ for $\Omega_{c}$. Let 
${I}=I_{\Omega}\cup I_{\Omega_{n}}\cup I_{\Omega_{c}}$ denote the union of total stencils within the entire domain (Interior and Boundary), and ${I}_B=I_{\Omega_{n}}\cup I_{\Omega_{c}}$ denote the stencils within the boundary regions $\Omega_n\cup\Omega_c$ (Boundary).

\vspace{.1in}
\noindent
Next we will firstly prove the positive-definiteness of $(-\mathcal{L}_{\delta,\Delta x}^{qnl})$ in Theorem~\ref{thm_DMP_static}, which is the discrete maximum principle for the static case; and then extend the result to the dynamic case in Theorem~\ref{thm_DMP_dynamic} where time derivative is involved.

\vspace{.1in}
\begin{theorem}\label{thm_DMP_static}
\textbf{Discrete Maximum Principle for the Static Case} The discrete operator $\mathcal{L}_{\delta,\Delta x}^{qnl}$ satisfies the maximum principle. For $u(x_i) \in \ell^1({I})$ with $\big(-\mathcal{L}_{\delta,\Delta x}^{qnl}\big)\left(u(x_j)\right)\leq0$ and $j\in I_{\Omega}$, and for any $i\in I=I_{\Omega}\cup I_{B}$, we have 
\begin{equation}\label{MaxP}
\max\limits_{i\in {I}}u(x_i)\leq \max\limits_{i\in {I}_B}u(x_i).
\end{equation}
Furthermore, equality holds, and $u(x_i)$ is a constant function on stencils $I$.
\end{theorem}

\begin{proof}
Suppose the discrete function $u$ achieves its strictly maximum values at an interior grid $j^*\in I_{\Omega}$.

\vspace{.25cm}
\noindent
\textbf{Case I Nonlocal}: Consider $j^*\in\{1,2,...,N\}$. Then since $u(x_{j^*})$ is a strict maximum
\begin{equation}\label{NLMaxP}
\mathcal{L}_{\delta,\Delta x}^{qnl}u_h(x_{j^*})=\sum_{k=1}^r\frac{u(x_{j^*+k})-2u(x_{j^*})+u(x_{j^*-k})}{(k\Delta x)^2}\int_{(k-1)\Delta x}^{k\Delta x}s^2\gamma_{\delta}(s)ds< 0
\end{equation}
which contradicts $-\mathcal{L}_{\delta,\Delta x}^{qnl}u(x_j^{*})\leq0$ unless $u$ is constant.

\vspace{.25cm}
\noindent
\textbf{Case II Transitional}:  Consider $j^* \in \{N+1,N+2,...,N+r\}$. We observe that 
\begin{equation}\label{TMaxP}
\int_{(k-1)\Delta x}^{k\Delta x}s^2\gamma_{\delta}(s)ds>(k-1)\Delta x\int_{(k-1)\Delta x}^{k\Delta x}s\gamma_{\delta}(s)ds.
\end{equation}
Using $u(x_{j^*})$ 
\begin{align}\label{TMaxP_2}
\mathcal{L}_{\delta,\Delta x}^{qnl}u_h(x_{j^*})&=\sum_{k=\frac{x_{j^*}}{\Delta x}+1}^r\frac{u(x_{j^*+k-1})-2u(x_{j^*})+u(x_{j^*-k+1})}{2(k-1)^2(\Delta x)^2}\int_{(k-1)\Delta x}^{k\Delta x}s^2\gamma_{\delta}(s)ds\nonumber\\
&-\sum_{k=\frac{x_{j^*}}{\Delta x}+1}^r\frac{u(x_{j^*+k-1})-u(x_{j^*-k+1})}{2(k-1)\Delta x}\int_{(k-1)\Delta x}^{k\Delta x}s\gamma_{\delta}(s)ds\nonumber\\
&+{\bigg(\int_{x_{j^*}}^{\delta}s\gamma_{\delta}(s)ds\bigg)\frac{u(x_{j^*+1})-u(x_{j^*})}{\Delta x}}\nonumber\\
&+\bigg(\int_0^{x_j^*}s^2\gamma_{\delta}(s)ds+x_{j^*}\int_{x_{j^*}}^{\delta}s\gamma_{\delta}(s)ds\bigg)\frac{u(x_{j^*+1})-2u(x_{j^*})+u(x_{j^*-1})}{(\Delta x)^2}.
\end{align}
Also since $u(x_{j^*})$ is a strict maximum we know
\begin{equation}\label{TMaxP_3}
\frac{u(x_{j^*+k-1})-2u(x_{j^*})+u(x_{j^*-k+1})}{2(k-1)^2(\Delta x)^2}< 0,
\end{equation}
combined with \eqref{TMaxP}, this gives us
\begin{align}\label{TMaxP_4}
\mathcal{L}_{\delta,\Delta x}^{qnl}u(x_{j^*})&\leq\sum_{k=\frac{x_{j^*}}{\Delta x}+1}^r\frac{u(x_{j^*+k-1})-2u(x_{j^*})+u(x_{j^*-k+1})}{2(k-1)^2(\Delta x)^2}\cdot(k-1)\Delta x\int_{(k-1)\Delta x}^{k\Delta x}s\gamma_{\delta}(s)ds\nonumber\\
&-\sum_{k=\frac{x_{j^*}}{\Delta x}+1}^r\frac{u(x_{j^*+k-1})-u(x_{j^*-k+1})}{2(k-1)\Delta x}\int_{(k-1)\Delta x}^{k\Delta x}s\gamma_{\delta}(s)ds\nonumber\\
&+\bigg(\int_{x_{j^*}}^{\delta}s\gamma_{\delta}(s)ds\bigg)\frac{u(x_{j^*+1})-u(x_{j^*})}{\Delta x}\nonumber\\
&+{\bigg(\int_0^{x_j^*}s^2\gamma_{\delta}(s)ds+x_{j^*}\int_{x_{j^*}}^{\delta}s\gamma_{\delta}(s)ds\bigg)\frac{u(x_{j^*+1})-2u(x_{j^*})+u(x_{j^*-1})}{(\Delta x)^2}}.
\end{align}

\noindent
By simplifying we conclude
\begin{align}\label{TMaxP_5}
\mathcal{L}_{\delta,\Delta x}^{qnl}u_h(x_{j^*})&\leq\sum_{k=\frac{x_{j^*}}{\Delta x}+1}^r\frac{{-2u(x_{j^*})}+2u(x_{j^*-k+1})}{2(k-1)\Delta x}\int_{(k-1)\Delta x}^{k\Delta x}s\gamma_{\delta}(s)ds\nonumber\\
&+\bigg(\int_{x_{j^*}}^{\delta}s\gamma_{\delta}(s)ds\bigg)\frac{u(x_{j^*+1})-{u(x_{j^*})}}{\Delta x}\nonumber\\
&+\bigg(\int_0^{x_j^*}s^2\gamma_{\delta}(s)ds+x_{j^*}\int_{x_{j^*}}^{\delta}s\gamma_{\delta}(s)ds\bigg)\frac{u(x_{j^*+1})-{2u(x_{j^*})}+u(x_{j^*-1})}{(\Delta x)^2}<0. 
\end{align}
which contradicts $\displaystyle{-\mathcal{L}_{\delta,\Delta x}^{qnl}u(x_j)\leq0}$.

\vspace{.25cm}
\noindent
\textbf{Case III Local}: Consider $j^*\in\{N+r+1,...,2N-1\}$. Then since $u(x_{j^*})$ is a strict maximum
\begin{equation}\label{LMaxP}
\mathcal{L}_{\delta,\Delta x}^{qnl}u(x_{j^*})=\frac{u(x_{j^*+1})-2u(x_{j^*})+u(x_{j^*-1})}{(\Delta x)^2}< 0
\end{equation}
\noindent
which contradicts $-\mathcal{L}_{\delta,\Delta x}^{qnl}u(x_j)\leq0$.
\end{proof}

\noindent
Next, we will consider the time-dependent case.

\begin{theorem}\label{thm_DMP_dynamic}
\textbf{Discrete Maximum Principle for the dynamic case}
Suppose for $i\in I=I_\Omega \cup I_B$ and $n=0,1,...,N_{T}-1$ with $T=N_T\cdot\Delta t$ that $\{u_i^n\}$ solves the following discrete QNL diffusion equation. 
\begin{equation}\label{discLtNDiffEq}
    \begin{cases}
     \frac{u_i^{n+1}-u_i^n}{\Delta t}=\mathcal{L}^{qnl}_{\delta,\Delta x}u_i^n+f_i^n, \hspace{.2in} for\hspace{.1cm} i\in I_{\Omega},\hspace{.1cm} and \hspace{.1cm} N_T>n\ge 0,\\
     u_i^0=g_i^0, \hspace{.2in} for \hspace{.1cm} i\in I \hspace{.1cm} \text{(Initial  Condition)},\\
     u_i^n=q_i^n, \hspace{.2in} for \hspace{.1cm} i\in I_B, \hspace{.1cm} n\geq 0 \hspace{.1cm}\text{(Boundary   Condition)},\\
   \end{cases}
\end{equation}

\noindent
then $u_i^n$ satisfies the discrete maximum principle 
\begin{equation}\label{DMP_dynamic_eq}
u_i^n\leq \max\{g_i^0|_{i\in I},\quad q_i^n|_{i\in I_B,n\geq0}\}
\end{equation}
given that $f_i^n\leq 0
$
for all $i\in I_{\Omega}$, all $n\geq0$, and $\frac{\Delta t}{\Delta x^2}\leq\frac{1}{4}$.

\end{theorem}

\begin{proof}
We denote $M=\max\{g_i^0|_{i\in I},\quad q_i^n|_{i\in I_B,n\geq0}\}$. Clearly, at $n=0$ we have $u_i^0\leq M$ for all $i\in I=I_\Omega \cup I_B$. We assume that this holds for $n=m$ with $0\leq m\leq N_T-2$. Now we would like to advance it to the next time level $n=m+1$.
\vspace{.1in}

\noindent
\textbf{Case I Nonlocal:} Consider $i\in \{1,2,...,N\}$ which is the nonlocal region. Then 
\begin{align*}
u_i^{m+1}&=u_i^m+\Delta t\bigg(\mathcal{L}^{qnl}_{\delta,\Delta x}u_i^m+f_i^m\bigg)\\
&\leq u_i^m+\Delta t \mathcal{L}^{qnl}_{\delta,\Delta x}u_i^m \nonumber\\
&= \bigg(1-\frac{2\Delta t}{\Delta x^2}\sum_{k=1}^r\frac{1}{k^2}\int_{(k-1)\Delta x}^{k\Delta x}s^2\gamma_{\delta}(s)ds\bigg)u_i^m + \frac{\Delta t}{\Delta x^2}\sum_{k=1}^r\frac{u_{i+k}^m+u_{i-k}^m}{k^2}\int_{(k-1)\Delta x}^{k\Delta x}s^2\gamma_{\delta}(s)ds.    
\end{align*}

\noindent
Notice that 
\begin{equation}\label{DMP_2}
\sum_{k=1}^r \frac{1}{k^2}\int_{(k-1)\Delta x}^{k\Delta x}s^2\gamma_{\delta}(s)ds\leq \sum_{k=1}^r\int_{(k-1)\Delta x}^{k\Delta x}s^2\gamma_{\delta}(s)ds=\int_0^{\delta}s^2\gamma_{\delta}(s)ds=1
\end{equation}
and $\displaystyle{\frac{\Delta t}{\Delta x^2}\leq \frac{1}{4}}$, so 
\begin{equation}\label{DMP_3}
\bigg(1-\frac{2\Delta t}{\Delta x^2}\sum_{k=1}^r\frac{1}{k^2}\int_{(k-1)\Delta x}^{k\Delta x}s^2\gamma_{\delta}(s)ds\bigg)\geq 0.
\end{equation}
Hence, 
\begin{align}\label{DMP_4}
u_i^{m+1}&\leq \bigg(1-\frac{2\Delta t}{\Delta x^2}\sum_{k=1}^r\frac{1}{k^2}\int_{(k-1)\Delta x}^{k\Delta x}s^2\gamma_{\delta}(s)ds\bigg)u_i^m+\frac{\Delta t}{\Delta x^2}\sum_{k=1}^r\frac{u_{i+k}^m+u_{i-k}^m}{k^2}\int_{(k-1)\Delta x}^{k\Delta x}s^2\gamma_{\delta}(s)ds \nonumber \\
&\leq \bigg(1-\frac{2\Delta t}{\Delta x^2}\sum_{k=1}^r\frac{1}{k^2}\int_{(k-1)\Delta x}^{k\Delta x}s^2\gamma_{\delta}(s)ds\bigg)M+\frac{\Delta t}{\Delta x^2}\sum_{k=1}^r\frac{M+M}{k^2}\int_{(k-1)\Delta x}^{k\Delta x}s^2\gamma_{\delta}(s)ds  \nonumber\\
&=M.
\end{align}

\noindent
\textbf{Case II Transitional:} Consider $i\in \{N+1,...,N+r\}$ which is the transitional region. Then 
\begin{align}\label{DMP_5}
u_i^{m+1}&\leq u_i^m+\Delta t \mathcal{L}^{qnl}_{\delta,\Delta x}u_i^m \nonumber \\
&=u_i^m+\Delta t\Bigg[\sum_{k=\frac{x_i}{\Delta x}+1}^r\frac{u_{i+k-1}^m-2u_i^m+u_{i-k+1}^m}{2(k-1)^2\Delta x^2}\int_{(k-1)\Delta x}^{k\Delta x}s^2\gamma_{\delta}(s)ds \nonumber \\
&-\sum_{k=\frac{x_i}{\Delta x}+1}^r\frac{u_{i+k-1}^m-u_{i-k+1}^m}{2(k-1)\Delta x}\int_{(k-1)\Delta x}^{k\Delta x}s\gamma_{\delta}(s)ds +\bigg(\int_{x_i}^{\delta}s\gamma_{\delta}(s)ds\bigg)\frac{u_{i+1}^m-u_i^m}{\Delta x} \nonumber \\
&+\bigg(\int_0^{x_i}s^2\gamma_{\delta}(s)ds+x_i\int_{x_i}^{\delta}s\gamma_{\delta}(s)ds\bigg)\frac{u_{i+1}^m-2u_i^m+u_{i-1}^m}{\Delta x^2}\Bigg] \nonumber \\
&=A\cdot u_i^m+\sum_{k=\frac{x_i}{\Delta x}+1}^r   \left(B_k\cdot u_{i+k-1}^m+C_k \cdot u_{i-k+1}^m+D\cdot u_{i+1}^m+E\cdot u_{i-1}^m \right)
\end{align}
where those notations are defined as
\begin{align}\label{DMP_6}
 A&=1+\frac{\Delta t}{\Delta x^2}\bigg(\sum_{k=\frac{x_i}{\Delta x}+1}^r\frac{-1}{(k-1)^2}\int_{(k-1)\Delta x}^{k\Delta x}s^2\gamma_{\delta}(s)ds\bigg)+\frac{\Delta t}{\Delta x}\bigg(-\int_{x_i}^{\delta}s\gamma_{\delta}(s)ds\bigg)\nonumber \\
 &-\frac{2\Delta t}{\Delta x^2}\bigg(\int_0^{x_i}s^2\gamma_{\delta}(s)ds+x_i\int_{x_i}^{\delta}s\gamma_{\delta}(s)ds\bigg),\nonumber \\
 B_k&=\frac{\Delta t}{2\Delta x^2(k-1)^2}\int_{(k-1)\Delta x}^{k\Delta x}s^2\gamma_{\delta}(s)ds-\frac{\Delta t}{2\Delta x(k-1)}\int_{(k-1)\Delta x}^{k\Delta x}s\gamma_{\delta}(s)ds, \nonumber \\
 C_k&=\frac{\Delta t}{2\Delta x^2(k-1)^2}\int_{(k-1)\Delta x}^{k\Delta x}s^2\gamma_{\delta}(s)ds+\frac{\Delta t}{2\Delta x(k-1)}\int_{(k-1)\Delta x}^{k\Delta x}s\gamma_{\delta}(s)ds, \nonumber \\
 D&=\frac{\Delta t}{\Delta x}\int_{x_i}^{\delta}s\gamma_{\delta}(s)ds+\frac{\Delta t}{\Delta x^2}\bigg(\int_0^{x_i}s^2\gamma_{\delta}(s)ds+x_i\int_{x_i}^{\delta}s\gamma_{\delta}(s)ds\bigg), and\nonumber \\
 E&=\frac{\Delta t}{\Delta x^2}\bigg(\int_0^{x_i}s^2\gamma_{\delta}(s)ds+x_i\int_{x_i}^{\delta}s\gamma_{\delta}(s)ds\bigg).
\end{align}

\noindent
Clearly, $\displaystyle{A+\sum_{k=\frac{x_i}{\Delta x}+1}^r(B_k+C_k)+D+E=1}$, and $B_k, C_k, D, E\geq0$ when $\Delta x$ is sufficiently small and 
because that $-\frac{\Delta t}{2\Delta x(k-1)}\int_{(k-1)\Delta x}^{k\Delta x}s\gamma_{\delta}(s)ds >
-\frac{\Delta t}{2(\Delta x)^2(k-1)^2}\int_{(k-1)\Delta x}^{k\Delta x}s^2\gamma_{\delta}(s)ds.
$

\vspace{0.1 in}
\noindent
Now we want to prove that $A\geq 0$. It is equivalent to prove 
\begin{align}\label{DMP_7}
AA&=\frac{\Delta t}{\Delta x^2}\Bigg[\sum_{k=\frac{x_i}{\Delta x}+1}^r\frac{1}{(k-1)^2}\int_{(k-1)\Delta x}^{k\Delta x}s^2\gamma_{\delta}(s)ds+2\bigg(\int_0^{x_i}s^2\gamma_{\delta}(s)ds+x_i\int_{x_i}^{\delta}s\gamma_{\delta}(s)ds\bigg) \nonumber\\
&+\Delta x\int_{x_i}^{\delta}s\gamma_{\delta}(s)ds\Bigg]\leq1. 
\end{align}

\noindent
Notice that 
\begin{align*}
AA&=\frac{\Delta t}{\Delta x^2}\Bigg[\sum_{k=\frac{x_i}{\Delta x}+1}^r\bigg(\frac{1}{(k-1)^2}\int_{(k-1)\Delta x}^{k\Delta x}s^2\gamma_{\delta}(s)ds+2x_i\int_{(k-1)\Delta x}^{k\Delta x}\bigg(\frac{1}{s}\bigg)s^2\gamma_{\delta}(s)ds\nonumber\\
&+\Delta x\int_{(k-1)\Delta x}^{k\Delta x}\bigg(\frac{1}{s}\bigg)s^2\gamma_{\delta}(s)ds\bigg)+2\int_0^{x_i}s^2\gamma_{\delta}(s)ds\Bigg]\nonumber\\
&\leq \frac{\Delta t}{\Delta x^2}\Bigg[\sum_{k=\frac{x_i}{\Delta x}+1}^r\bigg(\frac{1}{(k-1)^2}\int_{(k-1)\Delta x}^{k\Delta x}s^2\gamma_{\delta}(s)ds+\frac{2x_i}{(k-1)\Delta x}\int_{(k-1)\Delta x}^{k\Delta x}s^2\gamma_{\delta}(s)ds\nonumber\\
&+\frac{\Delta x}{(k-1)\Delta x}\int_{(k-1)\Delta x}^{k\Delta x}s^2\gamma_{\delta}(s)ds\bigg)+2\int_0^{x_i}s^2\gamma_{\delta}(s)ds\Bigg]\nonumber\\
&\leq\frac{\Delta t}{\Delta x^2}\Bigg[\sum_{k=\frac{x_i}{\Delta x}+1}^r4\int_{(k-1)\Delta x}^{k\Delta x}s^2\gamma_{\delta}(s)ds+4\int_0^{x_i}s^2\gamma_{\delta}(s)ds\Bigg]\nonumber\\
&=4\frac{\Delta t}{\Delta x^2}\Bigg[\sum_{k=\frac{x_i}{\Delta x}+1}^r\int_{(k-1)\Delta x}^{k\Delta x}s^2\gamma_{\delta}(s)ds+\int_0^{x_i}s^2\gamma_{\delta}(s)ds\Bigg]
=\frac{4\Delta t}{\Delta x^2}\int_0^{\delta}s^2\gamma_{\delta}(s)ds=4\frac{\Delta t}{\Delta x^2}.\leq 1 \hspace{.1cm}
\end{align*}
Since  $\frac{\Delta t}{\Delta x^2}\leq\frac{1}{4}$, so $AA\le 1$. Therefore, 

$A\geq0$ for $\displaystyle{B_k\geq\frac{\Delta t}{2\Delta x^2(k-1)^2}\int_{(k-1)\Delta x}^{k\Delta x}s^2\gamma_{\delta}(s)ds-\frac{\Delta t}{2\Delta x^2(k-1)^2}\int_{(k-1)\Delta x}^{k\Delta x}s^2\gamma_{\delta}(s)ds=0.}$

\vspace{.25in}
\noindent
Summarizing the coefficients of equation \eqref{DMP_5} gives

\begin{itemize}
    \item $A, B_k, C_k, D, E \geq0$
    \item $\displaystyle{A+\sum_{k=\frac{x_i}{\Delta x}+1}^r(B_k+C_k)+D+E=1}$.
\end{itemize}
\noindent
Hence $\displaystyle{u_i^{m+1}\leq \bigg(A+\sum_{k=\frac{x_i}{\Delta x}+1}^r(B_k+C_k)+D+E\bigg)M=M}.$

\vspace{.25cm}
\noindent
\textbf{Case III Local:} Consider $i\in \{N+r+1,...,2N-1\}$ which is the local region. Then

\begin{align*}
u_i^{m+1}&=u_i^m+\frac{\Delta t}{\Delta x^2}\bigg(u_{i+1}^m-2u_i^m+u_{i-1}^m\bigg)+\Delta t f_i^m\le\bigg(1-\frac{2\Delta t}{\Delta x^2}\bigg)u_i^m+\frac{\Delta t}{\Delta x^2}\bigg(u_{i+1}^m+u_{i-1}^m\bigg)
\end{align*}

\noindent
with $\frac{\Delta t}{\Delta x^2}\leq \frac{1}{4}$ which gives all positive coefficients, so $u_i^{m+1}\leq M$.

\vspace{.1in}
\noindent
Combining case I, II, III we can conclude that given $u_i^m\leq M$ for all $i\in I_{\Omega}$, and $\frac{\Delta t}{\Delta x^2}\leq \frac{1}{4}$ we have $u_i^{m+1}\leq M$ for all $i\in I_{\Omega}$. According to the induction we prove the theorem.

\end{proof}

\begin{corollary}
Suppose for $i\in I = I_{\Omega}\cup I_B$, $n=0,1,...,N_T-1,$ and $T=N_T\cdot \Delta t$ that $\{u_i^n\}$ solves the following discrete QNL diffusion equation \eqref{discLtNDiffEq} then we have the following upper bound for $u_i^n$ given that $\frac{\Delta t}{\Delta x^2}\leq \frac{1}{4}$,

\begin{equation}\label{DMP_10}
u_i^n\leq T\cdot ||f||_{\ell^{\infty}(I)}+max\{||g_i^0||_{\ell^{\infty}(I)},||q_i^n||_{\ell^{\infty}(I_B)}\}.
\end{equation}
\end{corollary}

\begin{proof}
We introduce a comparison function

\begin{equation}\label{DMP_11}
w_i^n= u_i^n+(T-n\cdot \Delta t)||f||_{\ell^{\infty}(I)}\geq u_i^n    
\end{equation}

\noindent
for $i\in I$, and $n\geq0$. Then we have

\begin{equation*}\label{DMP_12}
\frac{w_i^{n+1}-w_i^n}{\Delta t}=\frac{u_i^{n+1}-u_i^n}{\Delta t}-||f||_{\ell^{\infty}(I)}=\mathcal{L}^{qnl}_{\delta,\Delta x}u_i^n+\bigg(f_i^n-||f||_{\ell^{\infty}(I)}\bigg) 
\end{equation*}

\noindent
where $\bigg(f_i^n-||f||_{\ell^{\infty}(I)}\bigg)\leq 0$. Therefore by Theorem~\ref{thm_DMP_dynamic}, $w_i^n$ satisfies the discrete maximum principle $w_i^n\leq \max\{w_i^0|_{i\in I},\;\, w_i^n|_{i\in I_B}\}$  for all $i\in I_{\Omega}$ and $n\geq0$, given that $\frac{\Delta t}{\Delta x^2}\leq\frac{1}{4}$.

\noindent
Notice that 
\begin{equation}\label{DMP_13}
w_i^0=u_i^0+T\cdot ||f||_{\ell^{\infty}(I)}\leq \max\{||g_i^0||_{\ell^{\infty}(I)},||q_i^n||_{\ell^{\infty}(I_B)}\}+T\cdot ||f||_{\ell^{\infty}(I)} 
\end{equation}
and also that 
\begin{equation}\label{DMP_14}
w_i^n|_{i\in I_B}=u_i^n|_{i\in I_B}+\bigg(T-n\cdot \Delta t\bigg)||f||_{\ell^{\infty}(I)}\leq \max\{||g_i^0||_{\ell^{\infty}(I)},||q_i^n||_{\ell^{\infty}(I_B)}\}+T\cdot ||f||_{\ell^{\infty}(I)}.
\end{equation}
combined with the fact that $u_i^n|_{i\in I}\leq w_i^n|_{i\in I}$ proves the corollary.\end{proof}
\begin{remark}
Although in the proof of stability analysis, we require that $\frac{\Delta t}{\Delta x^2}\le \frac{1}{4}$ to proceed the analysis; meanwhile, we notice in the simulation that with $\frac{\Delta t}{\Delta x^2}$ close to $\frac{1}{2}$, we still have stable numerical results.
\end{remark}

\section{Convergence}\label{sec:convergence}
In this section, we prove the convergence results of the proposed FDM scheme.
\begin{theorem}
\textbf{Global error estimate of the discrete solution} Suppose $u(x,t)$ is the strong solution to \eqref{QNL_Cont} and $u_i^n$ is the discrete solution to the scheme \eqref{QNL_scheme} with $i\in I, n=0,1,...,N_T-1,$ and $N_T\Delta t=T$, respectively. Then we have
\begin{equation}\label{Conv}
|u(x_i,t^n)-u_i^n|\leq T\cdot C_{\delta}(\Delta x^2+\Delta t)
\end{equation}
given that $\frac{\Delta t}{\Delta x^2}\leq \frac{1}{4}$.
\end{theorem}

\begin{proof}
We define $e_i^n =u(x_i,t^n)-u_i^n$, $i=1,2,...,2N-1$, $n=0,1,...,N_T$ to be the error between the exact and discrete solutions. Then from the consistency analysis, and since $f_i^n=f(x_i,t^n)$ we have that 

\begin{equation}\label{Conv1}
    \begin{cases}
     \frac{e_i^{n+1}-e_i^n}{\Delta t}-\mathcal{L}_{\delta,\Delta x}^{qnl}e_i^n=\varepsilon_{c,i}, \hspace{.2in} for\hspace{.1cm} i\in I_{\Omega},\hspace{.1cm} and \hspace{.1cm} n\geq0\\
     e_i^0=0, i\in I\hspace{.2in}\text{(Initial Error)}\\
     e_i^n=0,\quad i\in I_{B} \hspace{.2in}\text{(Boundary \ Error)}\\
   \end{cases}
\end{equation}

\noindent
where $|\varepsilon_{c,i}|<C_\delta (\Delta x^2 + \Delta t)$ according to the consistency analysis. Hence we consider the following auxiliary function

\begin{equation}\label{Conv2}
w_i^n=e_i^n-(n\Delta t)\cdot C_\delta(\Delta x^2+\Delta t).
\end{equation}

\noindent
Observe that 
\begin{align}\label{Conv3}
\frac{w_i^{n+1}-w_i^n}{\Delta t}&-\mathcal{L}_{\delta,\Delta x}^{qnl}w_i^n \nonumber\\
&=\frac{{[e_i^{n+1}-C_\delta(\Delta x^2 +\Delta t)((n+1)\Delta t)]-[e_i^n-C_\delta(\Delta x^2+\Delta t)(n\Delta t)]}}{\Delta t}-\mathcal{L}_{\delta,\Delta x}^{qnl}e_i^n \nonumber\\ 
&= \frac{e^{n+1}_i-e_i^n}{\Delta t} -C_\delta(\Delta x ^2+\Delta t)-\mathcal{L}_{\delta,\Delta x}^{qnl}e_i^n\nonumber\\
&=\varepsilon_{c,i}-C_\delta(\Delta x ^2+\Delta t)\leq 0.
\end{align}
\noindent
Then $w_i^n$ satisfies
\begin{equation}\label{Conv4}
    \begin{cases}
     \frac{w_i^{n+1}-w_i^n}{\Delta t}-\mathcal{L}_{\delta,\Delta x}^{qnl}w_i^n\leq0, \quad i\in I_{\Omega},\\
     w_i^0=0, \quad i\in I,\hspace{.2in}\text{(Initial)},\\
     w_i^n=-(n\Delta t)\cdot C_\delta(\Delta x^2+\Delta t),\quad i\in I_{B}\hspace{.2in}\text{(Boundary)},\\
   \end{cases}
\end{equation}

\noindent
because of the the discrete maximum principle proved in Theorem~\ref{thm_DMP_dynamic}, so 
\begin{equation}\label{Conv5}
w_i^n\leq \max \{ w_i^0|{i\in I},\quad w_i^n|_{i\in I_B}\}=0, \quad \forall i\in I_{\Omega}.   
\end{equation}

\noindent
Therefore, $e_i^n\leq (n\Delta t)\cdot C_\delta(\Delta x^2+\Delta t)$. Similarly when $w_i^n=e_i^n+(n\Delta t)\cdot C_\delta(\Delta x^2+\Delta t)$ we have $e_i^n\geq -(n\Delta t)\cdot C_\delta(\Delta x^2+\Delta t)$. Hence, $|e_i^n|\leq (n\Delta t)\cdot C_\delta(\Delta x^2 +\Delta t)$ which gives $|u(x_i,t^n)-u_i^n|\leq T\cdot C_{\delta}(\Delta x^2+\Delta t)$.

\end{proof}

\section{Study of the 
Courant–Friedrichs–Lewy (CFL) condition}\label{sec:CFL}
In this section, we study the CFL condition of the new finite difference scheme by employing the Von Neumann stability analysis.
We denote $\frac{\Delta t}{\Delta x}$ by $\lambda_1$
and $\frac{\Delta t}{(\Delta x)^2}$ by $\lambda_2$
and insert $u_i^n=\left(g(\theta)\right)^n e^{\sqrt{-1}\theta  x_i}$  into the scheme \eqref{QNL_FDM} where $k$ is a given wave number. We get the following three different cases:
\begin{itemize}
\item \textbf{Case I Nonlocal}: for $x_i\leq0$, the growth factor is 
    \begin{equation}\label{growth_nonlocal}
        g(\theta)=1+\lambda_2\sum_{j=1}^{r} \frac{2\big(\cos(\theta j\Delta x)-1\big)}{j^2}\int_{(j-1)\Delta x}^{j\Delta x} s^2 \gamma_{\delta}(s)ds.
    \end{equation}
\item \textbf{Case II Transitional}: for $0<x_i\leq \delta$,
the growth factor is
\begin{equation}\label{growth_transition}
    \begin{split}
    g(\theta)=&1+\lambda_1\sum_{j=\frac{x_i}{\Delta x}+1}^{r}
    \frac{\big(\cos(\theta(j-1)\Delta x)-1\big)}{(j-1)}
    \int_{(j-1)\Delta x}^{j\Delta x}s\gamma_{\delta}(s)ds\\
    &-\lambda_1\sum_{j=\frac{x_i}{\Delta x}+1}^{r}
    \frac{\sqrt{-1}\sin (\theta(j-1)\Delta x)}{(j-1)}\int_{(j-1)\Delta x}^{j\Delta x} s\gamma_{\delta}(s)ds\\
    &+\lambda_1\left(\int_{x_i}^\delta s\gamma_\delta(s)ds\right)\left(\cos(\theta\Delta x)+\sqrt{-1}\sin(\theta\Delta x)-1\right)\\
    &+ \lambda_2\left(\int_0^{x_i}s^2\gamma_{\delta}(s)ds+
    x_i\int_{x_i}^{\delta} s\gamma_{\delta}(s)ds\right)\big(2\cos(\theta\Delta x)-2\big).
\end{split}
\end{equation}

\item \textbf{Case III Local}: for $x_i >\delta$, the growth factor is
\begin{equation}\label{growth_local}
    g(\theta)=1+\lambda_2\big(2\cos(\theta\Delta x)-2\big).
\end{equation}
\end{itemize}

\begin{proof}
Performing Von Nuemman analysis for stability we substitute $u_i^n=\left(g(\theta)\right)^n e^{\sqrt{-1}\theta x_i}$

\vspace{.25cm}
\noindent
\textbf{Case I}: 
\begin{equation}\label{CFL_1}
\frac{u_i^{n+1}-u_i^n}{\Delta t}=\sum_{j=1}^{r}\frac{u_{i+j}^n-2u_{i}^n+u_{i-j}^n}{(j\Delta x)^2}\int_{(j-1)\Delta x}^{j\Delta x}s^2\gamma_{\delta}(s)ds
\end{equation}
Substituting $u_i^n=\left(g(\theta)\right)^n e^{\sqrt{-1}\theta x_i}$ gives 
\begin{equation}\label{CFL_2}
g(\theta)^ne^{\sqrt{-1}\theta x_i}(g(\theta)-1)=\lambda_2\sum_{j=1}^{r}\frac{g(\theta)^ne^{\sqrt{-1}\theta x_i}\big(e^{\sqrt{-1}\theta\Delta x}-2+e^{-\sqrt{-1}\theta\Delta x}\big)}{j^2}\int_{(j-1)\Delta x}^{j\Delta x}s^2\gamma_{\delta}(s)ds.
\end{equation}
Therefore, we can conclude the growth factor for the nonlocal region is
\begin{equation}\label{CFL_3}
g(\theta)=1+\lambda_2\sum_{j=1}^{r}( \frac{2\big(\cos(\theta j\Delta x)-1\big)}{j^2}\int_{(j-1)\Delta x}^{j\Delta x} s^2 \gamma_{\delta}(s)ds.
\end{equation}

\noindent
\textbf{Case II}: 
\begin{align}\label{CFL_4}
\frac{u_i^{n+1}-u_i^n}{\Delta t}=&\sum_{j=\frac{x_i}{\Delta x}+1}^r\frac{u_{i+j-1}^n-2u_i^n+u_{i-j+1}^n}{2(j-1)\Delta x}\int_{(j-1)\Delta x}^{j\Delta x}s\gamma_{\delta}(s)ds\nonumber\\
&-\sum_{j=\frac{x_i}{\Delta x}+1}^r\frac{u_{i+j-1}^n-u_{i-j+1}^n}{2(j-1)\Delta x}\int_{(j-1)\Delta x}^{j\Delta x}s\gamma_{\delta}(s)ds\nonumber\\
&+\bigg(\int_{x_i}^{\delta}s\gamma_{\delta}(s)ds\bigg)\frac{u_{i+1}^n-u_i^n}{\Delta x}\nonumber\\
&+\bigg(\int_0^{x_i}s^2\gamma_{\delta}(s)ds+x_i\int_{x_i}^{\delta}s\gamma_{\delta}(s)ds\bigg)\frac{u_{i+1}^n-2u_i^n+u_{i-1}}{(\Delta x)^2}. 
\end{align}

\noindent
Similarly to the nonlocal region substituting $u_i^n=\left(g(\theta)\right)^n e^{\sqrt{-1}\theta x_i}$ gives
\begin{align}\label{CFL_5}
g(\theta)^n&e^{\sqrt{-1}\theta x_i}(g(\theta)-1)=\nonumber\\
&\lambda_1\sum_{j=\frac{x_i}{\Delta x}+1}^r\frac{1}{2(j-1)}\bigg(g(\theta)^ne^{\sqrt{-1}\theta x_i}\big(e^{\sqrt{-1}\theta(j-1)\Delta x}-2+e^{-\sqrt{-1}\theta(j-1)\Delta x}\big)\bigg)\int_{(j-1)\Delta x}^{j\Delta x}s\gamma_{\delta}(s)ds\nonumber\\
&-\lambda_1\sum_{j=\frac{x_i}{\Delta x}+1}^r\frac{1}{2(j-1)}\bigg(g(\theta)^ne^{\sqrt{-1}\theta x_i}\big(e^{\sqrt{-1}\theta(j-1)\Delta x}-e^{-\sqrt{-1}\theta(j-1)\Delta x}\big)\bigg)\int_{(j-1)\Delta x}^{j\Delta x}s\gamma_{\delta}(s)ds\nonumber\\
&+\lambda_1\bigg(\int_{x_i}^{\delta}s\gamma_{\delta}(s)ds\bigg)\bigg(g(\theta)^ne^{\sqrt{-1}\theta x_i}\big(e^{\sqrt{-1}k\Delta x}-1\big)\bigg)\nonumber\\
&+\lambda_2\bigg(\int_0^{x_i}s^2\gamma_{\delta}(s)ds+x_i\int_{x_i}^{\delta}s\gamma_{\delta}(s)ds\bigg)\bigg(g(\theta)^ne^{\sqrt{-1}\theta x_i}\big(e^{\sqrt{-1}\theta \Delta x}-2+e^{-\sqrt{-1}\theta\Delta x}\big)\bigg).
\end{align}

\noindent
Therefore, we can conclude the growth factor for the transitional region is
\begin{align}\label{CFL_6}
g(\theta)=&1+\lambda_1\sum_{j=\frac{x_i}{\Delta x}+1}^{r}
\frac{\big(\cos(\theta(j-1)\Delta x)-1\big)}{(j-1)}
\int_{(j-1)\Delta x}^{j\Delta x}s\gamma_{\delta}(s)ds\nonumber\\
&-\lambda_1\sum_{j=\frac{x_i}{\Delta x}+1}^{r}
\frac{\sqrt{-1}\sin (\theta(j-1)\Delta x)}{(j-1)}\int_{(j-1)\Delta x}^{j\Delta x} s\gamma_{\delta}(s)ds\nonumber\\
&+\lambda_1\left(\int_{x_i}^\delta s\gamma_\delta(s)ds\right)\left(\cos(\theta\Delta x)+\sqrt{-1}\sin(k\Delta x)-1\right)\nonumber\\
&+ \lambda_2\left(\int_0^{x_i}s^2\gamma_{\delta}(s)ds+
x_i\int_{x_i}^{\delta} s\gamma_{\delta}(s)ds\right)\big(2\cos(\theta\Delta x)-2\big).
\end{align}

\noindent
\textbf{Case III}: 
\begin{equation}\label{CFL_7}
\frac{u_i^{n+1}-u_i^n}{\Delta t}=\frac{u_{i+1}^n-2u_{i}^n+u_{i-1}^n}{(\Delta x)^2}
\end{equation}

\noindent
Finally, substituting $u_i^n=\left(g(\theta)\right)^n e^{\sqrt{-1}\theta x_i}$ gives
\begin{equation}\label{CFL_8}
g(\theta)^ne^{\sqrt{-1}\theta x_i}(g(\theta)-1)=\lambda_2\bigg(g(\theta)^ne^{\sqrt{-1}\theta x_i}\big(e^{\sqrt{-1}\theta\Delta x}-2+e^{-\sqrt{-1}k\Delta x}\big)\bigg).
\end{equation}

\noindent
Therefore, we can conclude the growth factor for the local region is
\begin{equation}\label{CFL_9}
g(\theta)=1+\lambda_2\big(2\cos(\theta\Delta x)-2\big).
\end{equation}

\noindent
Clearly, we have $\lambda_2=\Delta x \lambda_1$, so once we get the CFL constraint on $\lambda_1$, the CFL condition for $\lambda_2$ will be satisfied when $\Delta x$ is sufficiently small.
Because it is very difficult to analytically find this upper bound we implement the growth factor $g(\theta)$ numerically to identify restrictions on $\lambda_1$ and $\lambda_2$ to ensure $|g(\theta)|\le 1$. 
\end{proof}

\begin{figure}[htp!]
\centering
{\includegraphics[width=0.45\textwidth]{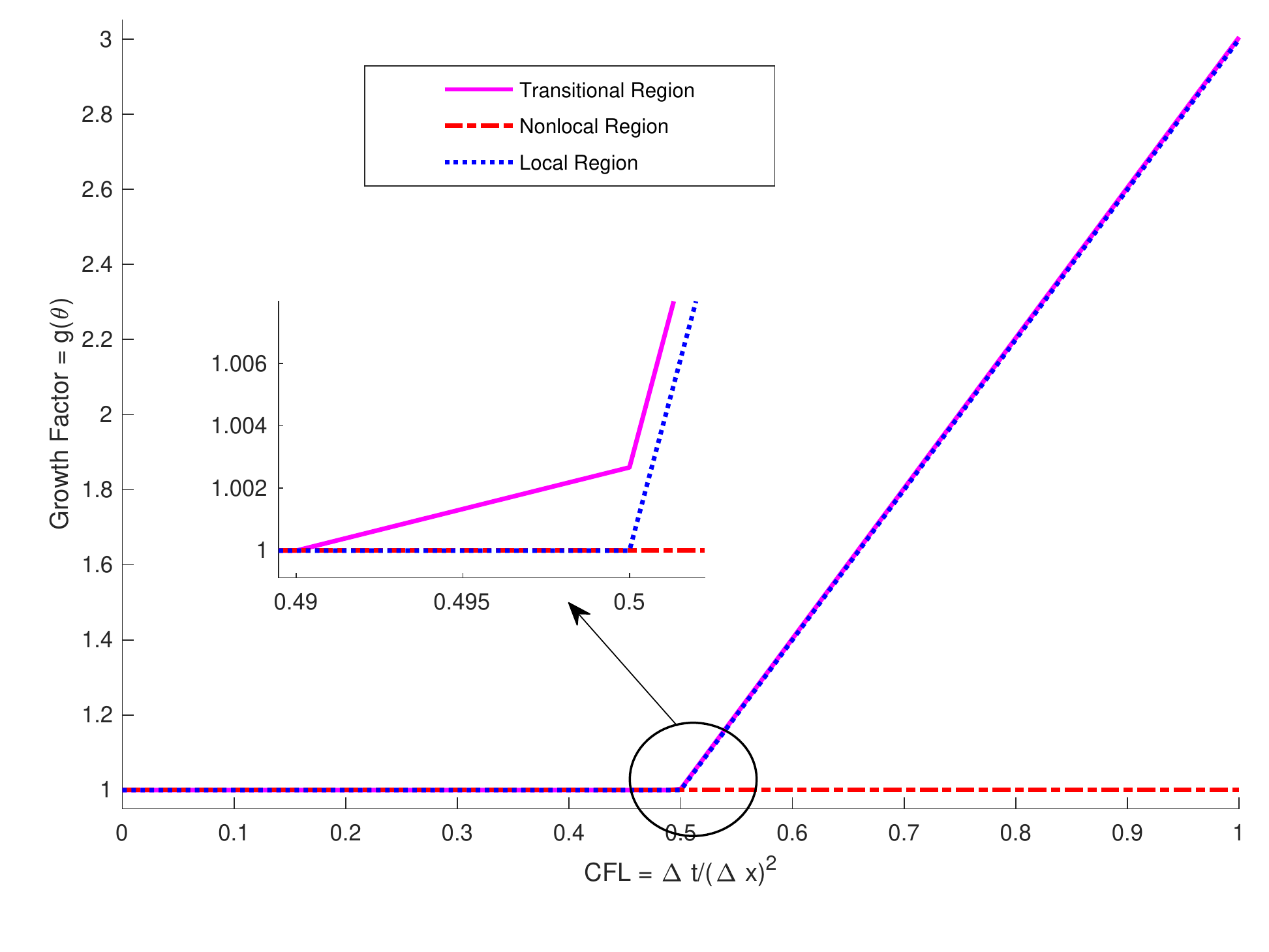}}
\label{fig:f8}
\caption{Maximum Growth Rate of \eqref{growth_nonlocal}, \eqref{growth_transition}, \eqref{growth_local}.\label{Fig:CFL}}
\end{figure}

\noindent
For linear local diffusion models with the explicit Euler and middle point finite difference discretization, the CFL is restricted by CFL $=\frac{\Delta t}{\Delta x^2}\le 0.5$. This provides the largest step size in time to reduce computational cost while preserves stability. By numerically analyzing the growth factor in Figure~\ref{Fig:CFL}, we found that the nonlocal and local regions match the typical restrictions for stability, but the transitional region is slightly less than 0.5. This factor needs to be considered for stability restrictions to the CFL on the whole coupling system. 

\section{Numerical Examples}\label{sec:numerics}
In this section, we test several numerical examples to confirm the stability and convergence results.

We fix the nonlocal diffusion kernel to be constant kernel
\[
\gamma_{\delta}(s)=\frac{3}{\delta^3}\chi_{[-\delta, \,\delta]}(s).
\]

\begin{enumerate}
\item For the first example, we consider the
asymptotic compatibility (AC) of the discretized operator $\mathcal{L}^{qnl}_{\delta, \Delta x}$
to the local diffusion problem as the horizon $\delta$ and spatial discretization $\Delta x$ go to zero at the same time.

We consider the external force $f$ as
\begin{equation}\label{Ex_1}
f(x, t) = 30x^4e^{-t} + e^{-t}(x^6 - 1) + 2.
\end{equation} 
Then, the exact solution to the local diffusion 
${u^{\ell}}_t=u^{\ell}_{xx}+f$ with $u^{\ell}(-1,t)=u^{\ell}(1,t)=0$ 
and $u^{\ell}(x,0)=(1-x^2)-(x^6-1)$
is
\begin{equation}\label{Ex_2}
u^{\ell}(x, t) = (1 - x^2) - e^{-t}(x^6 - 1).
\end{equation}

To test the AC convergence, we fix $\delta=r\Delta$ with $r=3$ and set the CFL to be $CFL=0.45$, that is $\Delta t=0.2 \Delta x$, and the termination time is chosen to be $T=1$. 

\noindent
First order convergence with respect to $\Delta x$ is observed. The convergence order and $L^{\infty}_{\Omega\times [0, T]}$ differences between $u^{\ell}(x,t)$ and discrete solution of $u^{qnl}_{\delta,\Delta x}$ are listed in Table \ref{table:1}. Also the visual comparison of the two solutions at $t=0$ and $t=T$ are displayed in Figure \ref{fig:2} with a nice agreement.

\begin{table}[H]
\centering
\begin{tabular}{|c|c|c|}
\hline
$\Delta x$ & $||u^{\ell}(x_i,t^n)-u_{\delta,\Delta x}^{qnl}(x_i,t^n)||_{L^{\infty}_{\Omega\times [0, T]}}$ & Order \\ \hline
$\frac{1}{50}$ & $0.1422$ & $-$ \\ \hline
$\frac{1}{100}$ & $7.168$e$-2$ & $0.988$ \\ \hline
$\frac{1}{200}$ & $3.614$e$-2$ & $0.988$\\ \hline
$\frac{1}{400}$ & $1.820
$e$-2$ & $0.990$\\ \hline
$\frac{1}{800}$ & $9.151
$e$-3$ & $0.992$\\ \hline
$\frac{1}{1600}$ & $4.594
$e$-3$ & $0.994$\\ \hline
\end{tabular}
\caption{$L^{\infty}_{\Omega\times [0, T]}$ differences between the local continuous solution $u^{\ell}$ and discrete solution $u_{\delta, \Delta x}^{qnl}$. We fix $\delta = 3\Delta x$, and the kernel is $\gamma_{\delta}(s)=\frac{3}{\delta^3}
\chi_{[-\delta,\delta]}(s)$. The termination time $T=1$ and $\Delta t=0.2 \Delta x$.}
\label{table:1}
\end{table}

\begin{figure}[H]
\centering
\subfigure[solutions at $t=0$]{{\includegraphics[width=0.42\textwidth]{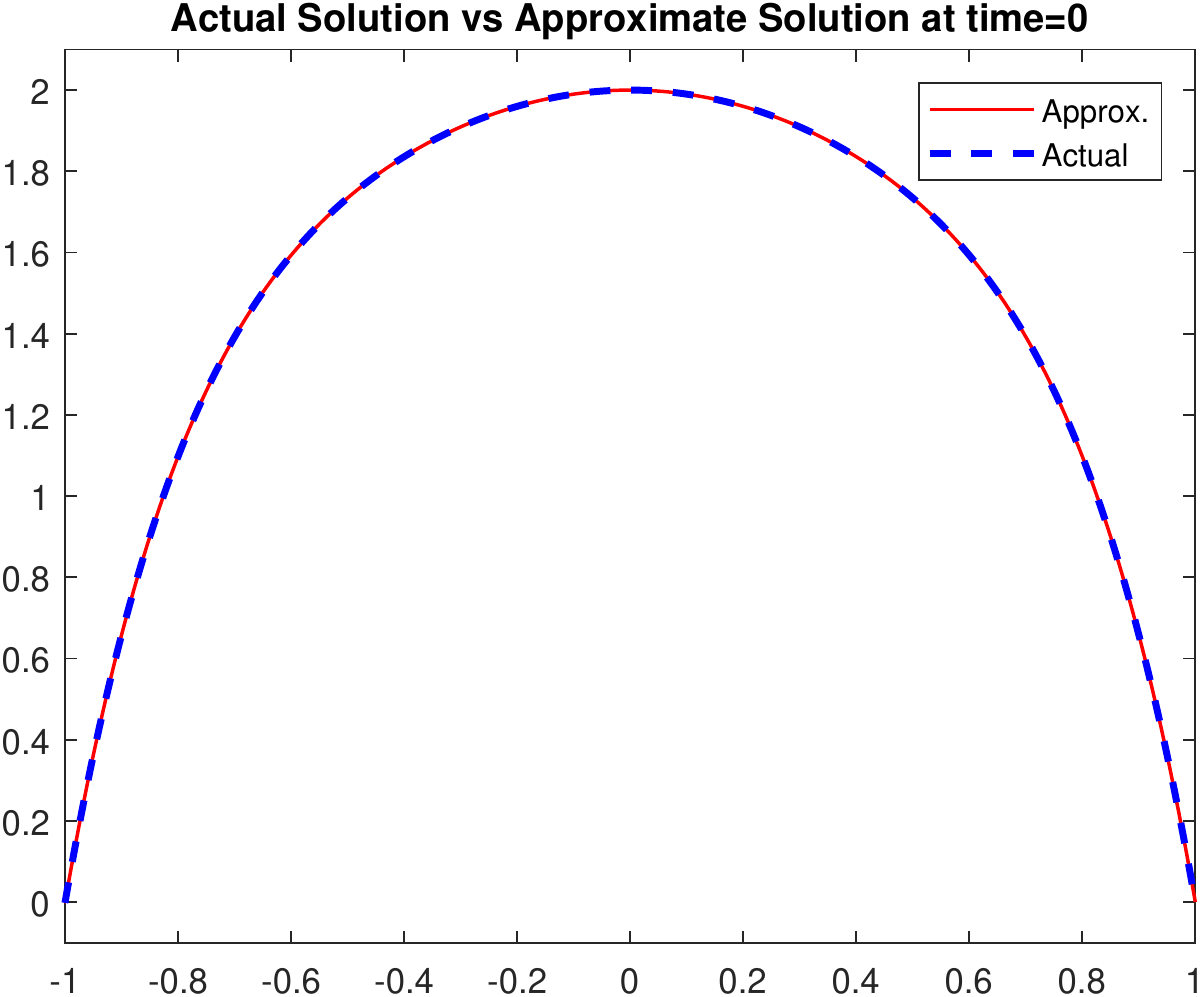}}}
\subfigure[solutions at $t=1$]{{\includegraphics[width=0.4\textwidth]{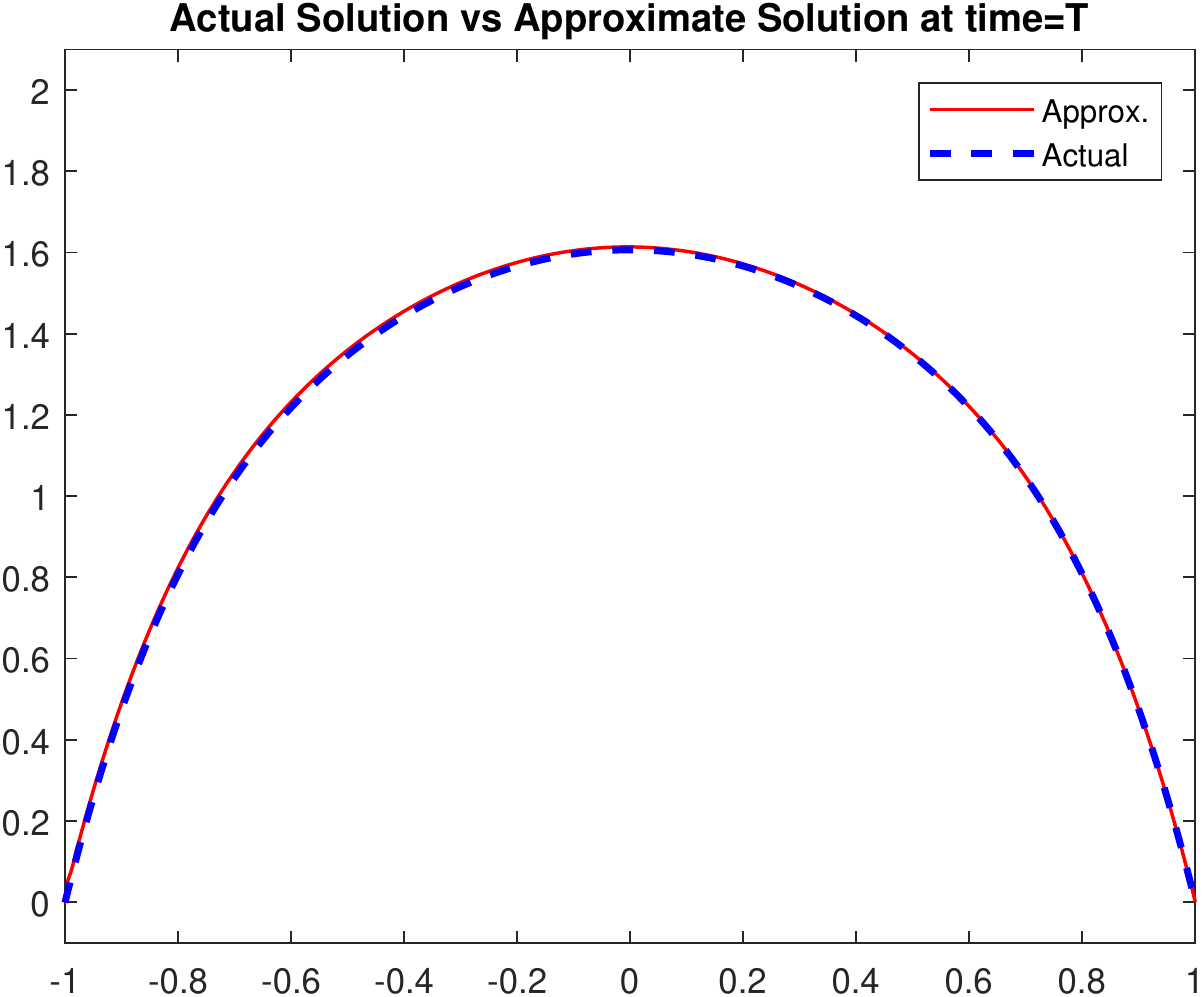}}}
\caption{Plots of solutions to the approximate and actual solutions. The kernel function was chosen as $\gamma_{\delta}(s)=\frac{3}{\delta^3}\chi_{[-\delta, \delta]}(s)$. The coupling inference is at $x^*=0$, and the mesh size is $\Delta x=\frac{1}{400}$ with a horizon as $\delta =\frac{3}{400}$, the temporal step size is $\Delta t=0.45\Delta x$.}
\label{fig:2}
\end{figure}

\item 
In the following example, we compare 
the original scheme  $\mathcal{\tilde{L}}_\delta^{qnl}$proposed in \cite{du2018quasinonlocal} with the new proposed scheme $\mathcal{L}_{\delta,\Delta x}^{qnl}$ in \eqref{QNL_FDM}. In \cite{du2018quasinonlocal}, the time-integral is still approximated by the explicit Euler method, and the  $\mathcal{\tilde{L}}_{\delta, \Delta x}^{qnl}$ is approximated by the following finite difference scheme given interface at $x^*=0$:
\begin{equation}\label{QNL_FDM2}
\mathcal{\tilde{L}}_{\delta,\Delta x}^{qnl}u_i^n\approx \begin{cases}
      \displaystyle{2\sum_{j=1}^r\frac{u_{i+j}^n-2u_i^n+u_{i-j}^n}{(j\Delta x)^2}\int_{(j-1)\Delta x}^{j\Delta x}s^2\gamma_{\delta}(s)ds}, \hspace{.25in} \text{if} \hspace{.1in} x_i<0.\\
      \\
      \displaystyle{\sum_{j=\frac{x_{i}}{\Delta x}}^{r}\frac{u_{i+j}^n-2u_{i}^n+u_{i-j}^n}{({j}\Delta x)^2}\int_{(j-1)\Delta x}^{j\Delta x}s^2\gamma_{\delta}(s)ds}\\
      \displaystyle{-\sum_{j=\frac{x_{i}}{\Delta x}}^{r}\frac{u_{i+j}^n-u_{i-j}^n}{{j}\Delta x}\int_{(j-1)\Delta x}^{j\Delta x}s\gamma_{\delta}(s)ds}\\
      \displaystyle{+2\bigg{(}\int_{x_{i}}^{\delta}s\gamma_{\delta}(s)ds\bigg{)}\frac{u_{i+1}^n-u_{i}^n}{\Delta x}}\\
      \displaystyle{+\bigg{(}2\int_{0}^{x_{i}}s^2\gamma_{\delta}(s)ds+2x_{i}\int_{x_{i}}^{\delta}s\gamma_{\delta}(s)ds\bigg{)}\frac{u_{i+1}^n-2u_{i}^n+u_{i-1}^n}{(\Delta x)^2}},\hspace{.2in}\text{if} \hspace{.1in} x_i\in[0,\delta],\\
      \\
      \displaystyle{\frac{u_{i+1}^n-2u_{i}^n+u_{i-1}^n}{(\Delta x)^2}}, \hspace{1.55in}\text{if}\hspace{.1in} x_i\in(\delta,1).\\
   \end{cases}
\end{equation}
Compare \eqref{QNL_FDM} with \eqref{QNL_FDM2}, we notice that the difference is replacing $j$ in the original scheme by $(j-1)$ in the new scheme. This is the main difference in the approximation that allows the equation \eqref{QNL_FDM}  to satisfy the discrete maximum principle where equation \eqref{QNL_FDM2} does not. 

Next, we are going to compare the AC convergence between \eqref{QNL_FDM}
and \eqref{QNL_FDM2}.
The exact local continuous solution is chosen to be
\begin{equation}\label{EX_3}
u^{\ell}(x, t)= e^{-t}(1-x)^2(1+x)^2x^2  
\end{equation}
and the corresponding external force is
\begin{equation}\label{EX_4}
\begin{split}
f(x,t) = &u^{\ell}_t-u^{\ell}_{xx}\\
=&-e^{-t}\left((x-x^3)^2+(2-24x^2+30x^4)\right).
\end{split}
\end{equation}

\noindent
 Again the kernel used is $\gamma_{\delta}(s) = \frac{3}{\delta^3}$ with $\delta =3\Delta x$. 
 We denote the solution obtained by $\mathcal{L}^{qnl}_{\delta, \Delta x}$ by $u^{qnl}_{\delta, \Delta x}$
 and the solution obtained by $\mathcal{\tilde{L}}^{qnl}_{\delta, \Delta x}$ by ${\tilde{u}}^{qnl}_{\delta, \Delta x}$.
 
 First order AC convergence with respect to $\Delta x$ are observed in Table~\ref{Table2} for both schemes \eqref{QNL_FDM} and \eqref{QNL_FDM2}, respectively. The approximation using scheme \eqref{QNL_FDM} at larger step size has second order convergence rate, and at smaller step size tends to be of first order.  
\begin{table}[H]
\centering
\begin{tabular}{|c|c|c|c|c|}
\hline
$\Delta x$ & $||{u}^{\ell}(x_i,t^n)-{\tilde{u}}_{\delta,\Delta x}^{qnl}(x_i,t^n)||_{L^{\infty}}$ & Order & $||{u}^{\ell}(x_i,t^n)-{{u}}_{\delta,\Delta x}^{qnl}(x_i,t^n)||_{L^{\infty}}$ & Order \\ \hline
$\frac{1}{50}$ & $9.255$e$-3$ & $-$ & $7.200$e$-3$ & $-$\\ \hline
$\frac{1}{100}$ & $4.692$e$-3$ & 0.980& $1.698$e$-3$ & 2.08 \\ \hline
$\frac{1}{200}$ & $2.356$e$-3$ & 0.994 & $4.121$e$-4$ & 1.09 \\ \hline
$\frac{1}{400}$ & 1.179e$-3$ & 0.998  & $1.931$e$-4$ & 1.09  \\ \hline
$\frac{1}{800}$ & 5.900e$-4$ & 0.999
&$9.628$e$-5$& 1.00
\\ \hline
$\frac{1}{1600}$ & 2.951e$-4$ & 1.00& 4.806e$-5$& $1.00$\\ \hline
\end{tabular}
\caption{$L^{\infty}_{\Omega\times [0,T]}$ differences between the local continuous solution $u^{\ell}$ and two discrete solutions $u_{\delta,\Delta x}^{qnl}$, ${\tilde{u}}_{\delta,\Delta x}^{qnl}$ using the FDM schemes \eqref{QNL_FDM} and \eqref{QNL_FDM2}, respectively. We fix $\delta = 3\Delta x$, and the kernel is $\gamma_{\delta}(s)=\frac{3}{\delta^3}$. The termination time is $T=1$ and $\Delta t=0.2 \Delta x$.\label{Table2}}
\end{table}

Next, we  compare the three solutions obtained from (1) new scheme; (2) exact local continuous solution and (3) the original scheme visually in Figure \ref{fig:3}. Notice that the exact local continuous solution $u^{\ell}(x,t)$ should remain non-negative throughout the entire computational domain $\Omega\times [0,T]$, however, both $u^{qnl}_{\delta, \Delta x}$ and 
${\tilde{u}}^{qnl}_{\delta, \Delta x}$ become slightly negative around the interface $x^{*}=0$. This does not contract the discrete maximum principle of $\mathcal{L}^{qnl}_{\delta, \Delta x}$ as the external force $f(x,t)$ defined in \eqref{EX_4} does not retain negative on $[-1, 1]$ as required in the assumption of Theorem~\ref{thm_DMP_dynamic}. On the other hand, because $\mathcal{L}^{qnl}_{\delta, \Delta x}$ satisfies the discrete maximum principle, consequently, $u^{qnl}_{\delta, \Delta x}$ provides less artificial negativity than ${\tilde{u}}^{qnl}_{\delta, \Delta x}$ around the interface of coupling.

\begin{figure}[h]
\centering
\subfigure[solutions at $t=0$]{\includegraphics[width=0.46\textwidth]{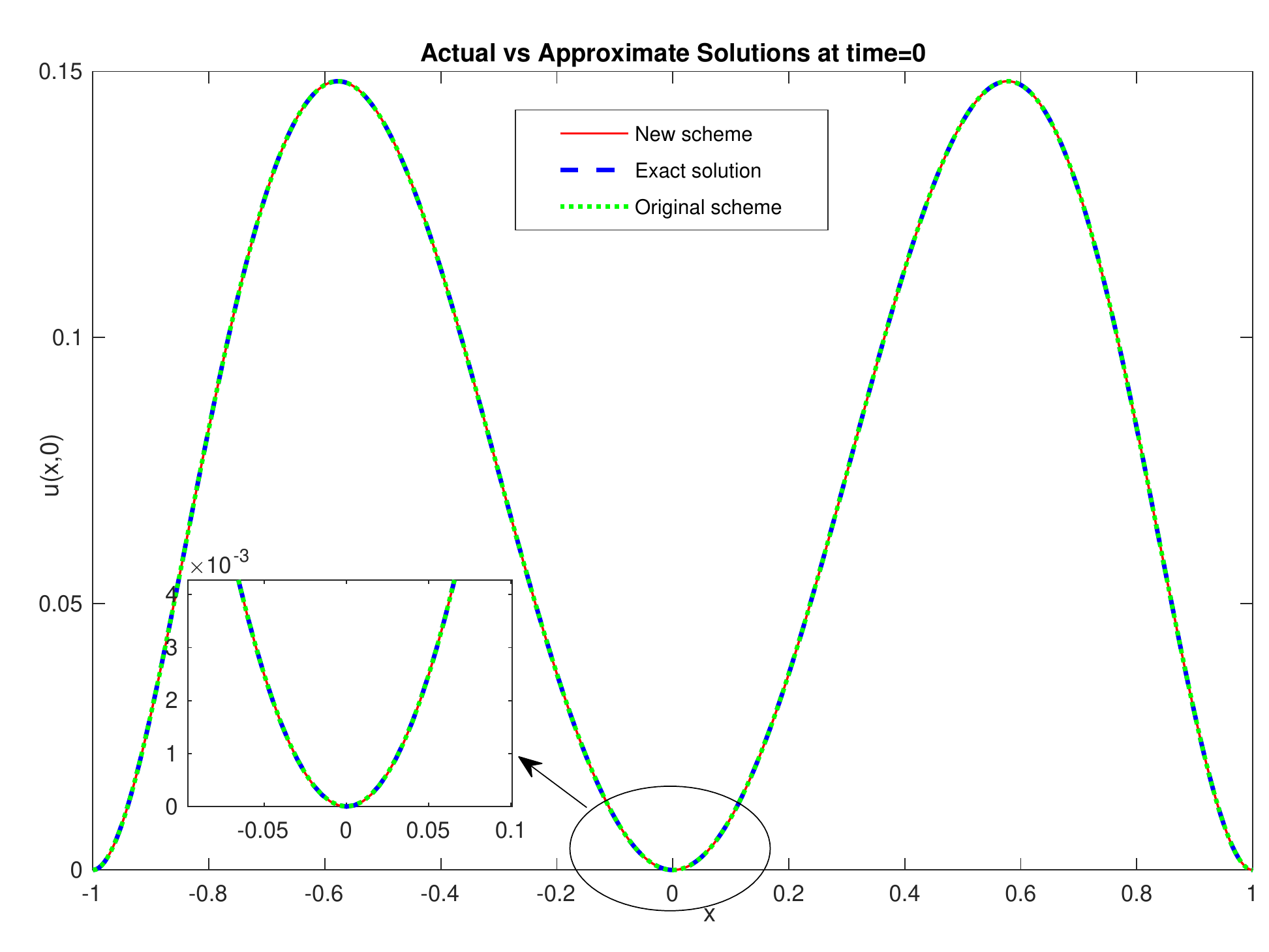}}
\subfigure[solutions at $t=1$]{\includegraphics[width=0.45\textwidth]{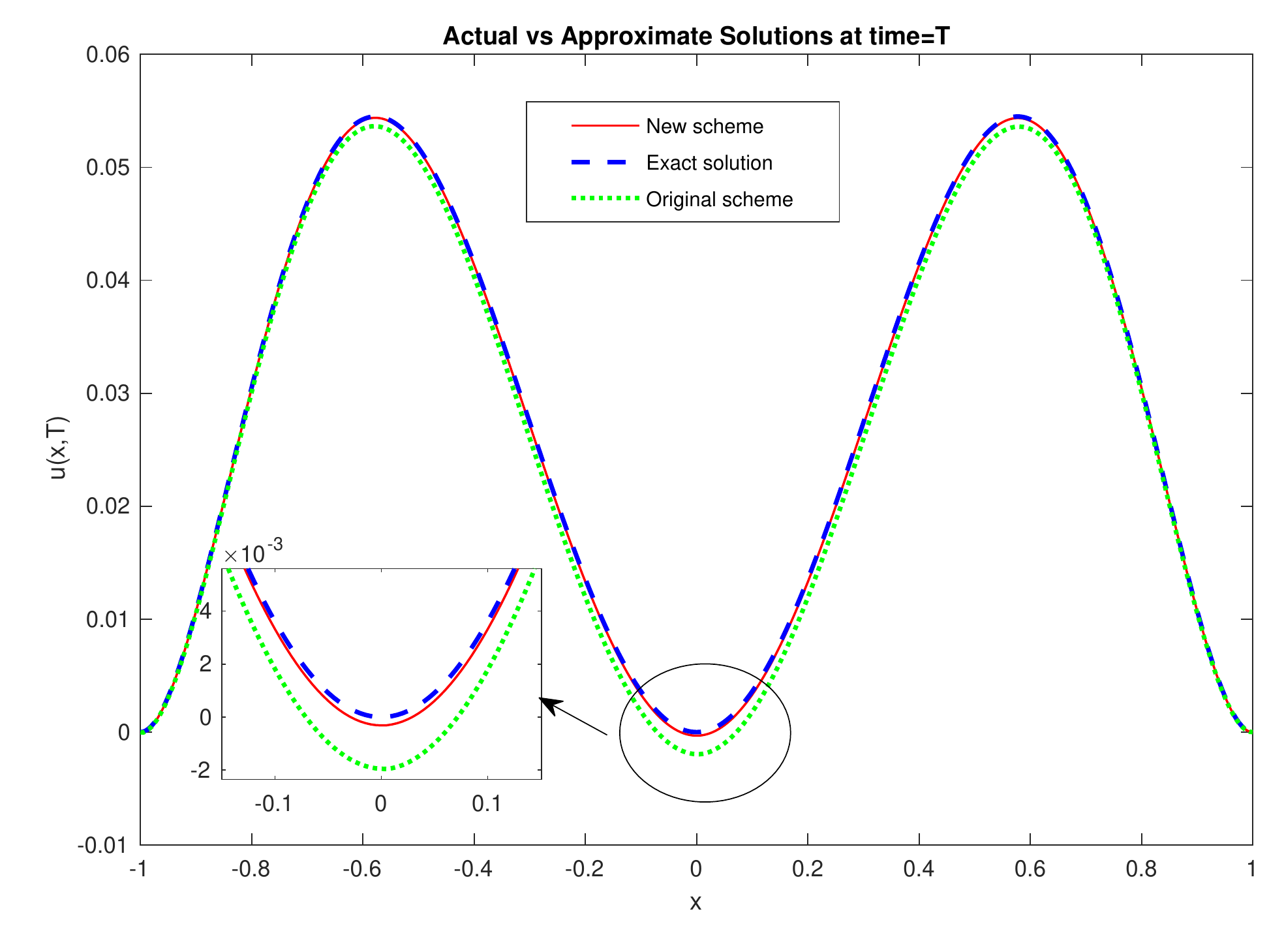} } 
\caption{Numerical comparison between the new scheme \eqref{QNL_FDM} and original scheme \eqref{QNL_FDM2} used to approximate \eqref{EX_3} with external force given by \eqref{EX_4}. The spatial step size is $\Delta x=\frac{1}{200}$ and $\Delta t$=$0.25\Delta x$. } 
\label{fig:3}
\end{figure}


\end{enumerate}
\section{Conclusion}
We propose a new scheme to discretize the quasi-nonlocal (QNL) coupling operator introduced in \cite{du2018quasinonlocal} for the nonlocal-to-local diffusion problem.
This new finite difference approximation preserves the properties of continuous equation on a discrete level. Consistency, stability, the maximum principle and the global convergence analysis of the scheme are proved rigorously. We analytically find the CFL conditions through the Von Neumann stability analysis and numerically calculate the CFL values for a given spatial discretization. The numerical calculations of the CFL provide us addition alert around the interface when considering the temporal step size for an explicit time integrator, as the CFL restrictions on the transitional region was discovered to be slightly less than $\frac{1}{2}$ with explicit Euler method employed in a diffusion problem. Multiple numerical examples are then provided and summarized to verify the theoretical findings. A comparison with the original scheme used in \cite{du2018quasinonlocal} is also provided which confirmed the improvements of the new scheme. 


\section{Acknowledgements}
 Amanda Gute and Dr. X. Li are supported by  NSF CAREER award: DMS-1847770 and the University of North Carolina at Charlotte Faculty Research Grant.

\bibliographystyle{abbrv}
\bibliography{QNLcouple}
\end{document}